\pdfoutput=1 

\documentclass[a4paper,american,reqno]{amsart}

\newif\ifPreprint \Preprinttrue
\newif\ifSubmission \Submissionfalse

\usepackage{babel}
\usepackage[utf8]{inputenc}
\usepackage[T1]{fontenc}
\usepackage{booktabs}
\usepackage{xspace}
\usepackage[binary-units=true]{siunitx}
\usepackage{etoolbox}
\robustify\bfseries
\usepackage{tabularx}
\usepackage{longtable}

\usepackage{lmodern} 
\usepackage{mathtools,amssymb} 
\usepackage{algorithm} 
\usepackage{algorithmic} 
\usepackage{tikz}
\usepackage{pgfplots}
\usepackage{pgffor}
\usetikzlibrary{intersections}
\usepackage{subcaption}
\usepackage{multirow}

\usepackage{placeins}
\usepackage{relsize}
\usepackage{caption}
\usepackage{todonotes}
\usepackage{csquotes}

\usepackage{microtype}
\usepackage[style = numeric-comp,
            maxbibnames = 100,
            maxcitenames = 2,
            giveninits = true,
            isbn = false,
            backend = biber]{biblatex}
\usepackage[colorlinks,
            citecolor=blue,
            urlcolor=blue,
            linkcolor=blue]{hyperref} 
\usepackage{orcidlink}
\usepackage{cleveref} 
\usepackage[nolist]{acronym}
\allowdisplaybreaks
\makeatletter
\patchcmd{\@settitle}{\uppercasenonmath\@title}{\scshape\large}{}{}
\patchcmd{\@setauthors}{\MakeUppercase}{\scshape\normalsize}{}{}
\makeatother

\vfuzz \hfuzz
\newcommand{\abbr}[1][abbrev]{#1.\xspace}

\newcommand{\eg}{\abbr[e.g]}

\newcommand{\ie}{\abbr[i.e]}

\newcommand{\st}{\text{s.t.}}

\newcommand{\T}{\mathsf{^T}}

\newcommand{\define}{\mathrel{{\mathop:}{=}}}
\newcommand{\enifed}{\mathrel{{=}{\mathop:}}}

\newcommand{\field}{\mathbb}
\newcommand{\naturals}{\field{N}}

\newcommand{\reals}{\field{R}}
\newcommand{\integers}{\field{Z}}
\newcommand{\N}{\naturals}

\newcommand{\R}{\reals}
\newcommand{\Z}{\integers}

\makeatletter
\newcommand{\fcdot}{\,\cdot\,}
\newcommand{\fcarg}[1]{\def\fc@rg{#1}\ifx\fc@rg\empty\fcdot\else\fc@rg\fi}
\makeatother
\newcommand{\abs}[1]{\lvert\fcarg{#1}\rvert}
\newcommand{\Abs}[1]{\left\lvert#1\right\rvert}
\newcommand{\norm}[2][]{\lVert\fcarg{#2}\rVert\ifx#1\empty\else_{#1}\fi}
\newcommand{\Norm}[2][]{\left\lVert#2\right\rVert\ifx#1\empty\else_{#1}\fi}

\DeclareMathOperator*{\argmax}{arg\,max}

\newcommand{\Set}[1]{\left\{#1\right\}}

\newcommand{\Defset}[3][\defsep]{\Set{#2#1#3}}

\newcommand{\eps}{\varepsilon}

\newtheorem{theorem}{Theorem}[section]
\newtheorem{lemma}[theorem]{Lemma}

\newtheorem{remark}[theorem]{Remark}

\crefname{assumption}{assumption}{assumptions}

\newcommand{\interior}[1]{{\kern0pt#1}^{\mathrm{o}}}

\newcommand{\vect}[1]{\mathbf{#1}} 
\newcommand{\domain}{\mathcal{D}}
\newcommand{\vx}{\vect{x}}
\newcommand{\va}{\vect{a}}
\newcommand{\vb}{\vect{b}}
\newcommand{\vv}{\vect{v}}
\newcommand{\vt}{\vect{t}}
\newcommand{\vd}{\vect{d}}
\newcommand{\sol}{\vx^*}

\newcommand{\software}[1]{\texttt{#1}}

\DeclareUnicodeCharacter{2014}{\xxx}
\DeclareRobustCommand\xxx{%
	\xxx\thinspace}

\usepackage[foot]{amsaddr}

\begin{document}

\title[Parabolic Approximation \& Relaxation for MINLP]
{
Parabolic Approximation \& Relaxation for MINLP
}

\author[A. Gö$\beta$, R. Burlacu, A. Martin]%
{Adrian Gö$\beta^\ast$\orcidlink{https://orcid.org/0009-0002-7144-8657}, Robert Burlacu\orcidlink{https://orcid.org/0000-0003-0578-6260}, Alexander Martin\orcidlink{https://orcid.org/0000-0001-7602-3653}
\\ \\ \today}

\address{$^\ast$Corresponding author}
\address[A. Gö$\beta$, R. Burlacu, A. Martin]{
		University of Technology Nuremberg (UTN), 
		\newline
		Analytics \& Optimization Lab,
		Dr.-Luise-Herzberg-Str.~4,
		90461~Nuremberg,
		Germany}
\email[A. Gö$\beta$, R. Burlacu, A. Martin]{$\{$adrian.goess, robert.burlacu, alexander.martin$\}$@utn.de}

\begin{abstract}
  We propose an approach based on quadratic approximations for solving general Mixed-Integer Nonlinear Programming (MINLP) problems.
Specifically, our approach entails the global approximation of the epigraphs of constraint functions by means of paraboloids, which are  polynomials of degree two with univariate quadratic terms, and relies on a Lipschitz property only.
These approximations are then integrated into the original problem.
To this end, we introduce a novel approach to compute globally valid epigraph approximations by paraboloids via a Mixed-Integer Linear Programming (MIP) model.
We emphasize the possibility of performing such approximations a-priori and providing them in form of a lookup table, and then present several ways of leveraging the approximations to tackle the original problem.
We provide the necessary theoretical background and conduct computational experiments on instances of the MINLPLib.
As a result, this approach significantly accelerates the
solution process of MINLP problems, particularly those involving many trigonometric or few exponential functions.
In general, we highlight that the proposed technique is able to exploit advances in Mixed-Integer Quadratically-Constrained Programming (MIQCP) to solve MINLP problems.

\end{abstract}

\keywords{mixed-integer nonlinear programming,
mixed-integer linear programming,
quadratic approximation,
global optimization\hspace{-1mm}
}

\maketitle\
\section*{Acknowledgments}
\label{sec:acknowledgements}

We thank the German Research Foundation (Deutsche Forschungsgemeinschaft, DFG) for their support within project A05 in the
\enquote{Sonderforschungsbereich/Transregio 154 Mathematical Modelling, Simulation and Optimization using the Example of Gas Networks} with Project-ID 239904186.
Further, we gratefully acknowledge the scientific support and HPC resources provided by the Erlangen National High Performance Computing Center (NHR@FAU) of the Friedrich-Alexander-Universität Erlangen-Nürnberg (FAU). The hardware is funded by the DFG.

\newpage
\section{Introduction}
\label{sec:introduction}

\noindent
We consider the problem
	\begin{subequations}
	\label{prob:minlp}
	\begin{align}
		\min \ & c(\vx) \\
		\st \ & f_j(\vx)  \leq y_j, \qquad j \in J, \\
		& \vx \in \Omega, \\
		& \vect{y} \in [\underline{\vect{y}}, \overline{\vect{y}}],
	\end{align}
\end{subequations}
where $\Omega = [\underline{\vx}, \overline{\vx}] \cap (\Z^p \times \R^{n-p}) \neq \emptyset$ for $p, \, n \in \N$ with $0 \leq p \leq n$ and $\abs{J} < \infty$.
We assume the functions $c, f_j : \R^n \to \R$ to be continuous on the domain $[\underline{\vx}, \overline{\vx}]$, $j \in J$.
Note that we can assume linearity of the objective $c$ by introducing a separate variable $z$, which is then minimized subject to $c(\vx) \leq z$ and the remaining constraints. 
We assume the existence of integer variables throughout the paper, \ie, $p > 0$, unless otherwise stated.
In general, we consider at least one function to be nonlinear and therefore we refer to~\labelcref{prob:minlp} as a \ac{minlp} problem or, simply, a~\ac{minlpm}. 
The right-hand sides $y_j$ are inserted for ease of presentation and do not pose a major limitation, as discussed later.

The solution methods for a \ac{minlpm} can vary enormously depending on the properties of the functions involved. 
See~\cite{Horst2013} for an overview of global solution strategies (mainly for continuous variables) and~\cite{Belotti2013} for a survey of the most common components of practical approaches for \ac{minlp} problems.

In the case that all functions involved are convex, the feasible set of the continuous relaxation of~\labelcref{prob:minlp} is also convex.
Consequently, this is typically referred to as a \textit{convex} \ac{minlp}. 
Such problems are usually tackled using cutting methods, such as outer approximation~\cite{Duran1986} or extended cutting-plane methods~\cite{Westerlund1995}, and decomposition schemes, such as generalized Benders decomposition~\cite{Geoffrion1972}.

However, if at least one of the functions is not convex, the situation changes, and~\labelcref{prob:minlp} is then referred to as a \textit{non-convex} \ac{minlp}.
Typically, non-convex \ac{minlp} problems are solved by a combination of convex relaxation of the non-convex nonlinearities and a refinement mechanism.
Effective convexification techniques are proposed by~\cite{Tawarmalani2013}, among others.
Since these may not provide a sufficiently tight representation of the original domain, the feasible region can be split to generate new subproblems with tighter convex relaxations.
This divide-and-conquer approach is performed in spatial branch-and-bound by reformulation (see, \eg,~\cite{Ryoo1996}) or in $\alpha$-branch-and-bound by $\alpha$-convexification~\cite{Androulakis1995}.
In the last decade, a viable alternative has received new attention, namely the approximation of the non-convex functions by piecewise linear functions and the subsequent reformulation as a \ac{mip} problem (alternatively, \ac{mipm}), particularly for real-world applications with considerable combinatorics~\cite{Burlacu2020,Geissler2012}.
For a comprehensive overview of the theoretical properties of the various \ac{mip}-based models for piecewise linear functions, refer to~\cite{Vielma2015}.

A natural generalization of the piecewise linear approach is the piecewise approximation by polynomials of a certain degree. 
Such an approximation is performed in an optimization setting using quadratic functions~\cite{Buchheim2013}.
In certain algebraic subfields, sets of more general polynomials of this type are usually referred to as \textit{splines}, and a distinction is made between approximations with fixed and variable knots, \ie, discretization points of an interval.
There is a vast amount of literature starting from the 1960s on splines. 
\citeauthor{Nurnberger1989}~\cite{Nurnberger1989} gives a comprehensive review of the first approximately 30 years.
For a more recent introduction to the topic, we refer the reader to~\cite{Schumaker2007}.

From an optimization perspective, the piecewise approximation by higher-order polynomials instead of linear functions does not offer a clear advantage, as this implies the introduction of both non-convex constraints and integer variables.
However, if we can perform a one-sided spline approximation, i.e., an underestimation of a function on its domain, this serves as a globally valid relaxation that does not require piecewise modeling.
In other terms, the underestimation is an approximation of the epigraph of the function.
In~\cite{Bojanic1966}, the authors prove existence and uniqueness for a best one-sided approximation with maximum degree. 
A computational procedure to compute such an approximation with a single polynomial is presented in~\cite{Lewis1970}.

Surprisingly, we did not find any literature that treats the one-sided approximation by a spline with variable knots, as we require it for optimization.
The existing research is either limited to spline approximations that are not one-sided (see~\cite{Ahlberg2016} or again~\cite{Schumaker2007}), which does not yield a globally valid relaxation of the \ac{minlpm}, or it focuses on best approximation by a single polynomial (see~\cite{Deikalova2020}).

In \Cref{sec:approx}, we thus present an \ac{mip} model for the one-sided approximation by univariate quadratic functions and prove its correctness.
In~\Cref{sec:gbm}, we formally introduce the resulting method for approximating non-convex constraint functions by a small number of such quadratic functions 
and their integration into an \ac{minlpm}.
This is complemented by computational showcases in~\Cref{sec:computations}, and the article concludes with remarks and further research directions in~\Cref{sec:conclusion}.

\subsection{Contribution}

We propose a framework to solve optimization problems with general nonlinear constraints by employing multiple quadratic constraints, which fulfill a predefined approximation accuracy. 
To the best of our knowledge, there is no approach for solving \ac{minlp} problems that globally approximates non-convex functions on their respective domains to a given accuracy without introducing additional variables, \ie, increasing the dimension of the problem, or branching on continuous variables.
In essence, this framework has the potential to turn an \ac{minlp} problem into an \ac{miqcp} one 
such that a solution to the latter serves as an $\eps$-approximate solution for the original problem.
Thus, this approach enables to leverage recent advancements in \ac{miqcp} (see, for instance,~\cite{Beach2022,Beach2024,Wiese2021}) to solve \ac{minlp} problems.

To this end, we present a novel approach to compute a globally valid one-sided approximation for nonlinear functions by means of polynomials of degree two with univariate quadratic terms only.
This includes an \ac{mip} model, which basically relies on a Lipschitz property only, and several strategies to determine the number of approximating functions.
Based on the results, we introduce techniques to incorporate such parabolic approximations in \ac{minlp} problems to provide dual bounds or even accelerate the solving of the original problem.
The theoretical explanations are complemented with computational experiments on the MINLPLIb~\cite{MINLPLib}, demonstrating limitations and potentials by this framework.

Lastly, we motivate the usage of a lookup table for approximations on generic domains when solving \ac{minlp} problems.
This is possible through the properties of global validity and one-sidedness.
Akin to the realm of Machine Learning, this accepts extensive a-priori computations and then compensates for them by repetitive use of the results.

\subsection{Notation}
A vector $\vx = (x_1, \dots,x_n)\T \in \R^n$ is denoted with bold and upright letters, whereas components $x_i$ are italic and not bold. 
A comparison between two vectors $\vx, \vect{z} \in \R^n$, \eg, $\vx \leq \vect{z}$, is always interpreted component-wise.
In the theoretical section, we make use of the 1-norm of a vector $\vx$ which is defined as $\norm[1]{\vx} = \sum_{i = 1}^n \abs{x_i}$.
For a domain $\domain \subseteq \R^n$, we define the set of continuous functions on $\domain$ as $\mathcal{C}^0(\domain)$. 
Further, $\mathcal{P}(\domain)$ denotes the power set of $\domain$.
Lastly, for a positive integer $m \in \N$, we state the set of indexes up to $m$ as $[m] \coloneqq \{1, \dots, m\}$. 

Note that we present our theoretical part for constraint functions on a multi-dimensional domain. 
We aim for their approximation by quadratic functions without bivariate terms and refer to such as \textit{paraboloids}.
Hence, the global one-sided approximation by paraboloids is called a \textit{parabolic approximation}.
Later on, when we substitute constraint functions with mentioned approximations, we name such a \textit{parabolic relaxation}, emphasizing the character of this procedure.

\section{Parabolic Approximation -- MIP Approach}
\label{sec:approx}

The term \emph{one-sided} implies an approximation of the function's (hypo-/)epigraph, coining the expression of an approximation \emph{from below (above)}.
In this context, a set of paraboloids serves as a \emph{global} approximation if every paraboloid is a underestimator(/overestimator) of the approximated function on its entire domain. 
In the following, we introduce an \ac{mip} model to compute such a global one-sided approximation given a fixed number $K \in \N$ of paraboloids and a Lipschitz function~$f$ together with its Lipschitz constant $L$.
We also include the usage of a measure to compute the Lebesgue integral of $f$ on a given domain. 
This is not necessary for the functionality of the model, but decreases its size, as mentioned in the respective paragraph.

Formally, let $f: \domain \mapsto \R$ be a Lipschitz continuous function with respect to $\norm[1]{\fcdot}$ and some $L > 0$,
where $\domain = [\va, \vb]$ a non-empty, full dimensional box defined by the vectors $\va, \, \vb \in \R^n$. 
That is, for all $\vx, \vect{y} \in \domain$, it holds true that
\begin{equation*}
	\abs{f(\vx) - f(\vect{y})} \leq L \norm[1]{\vx - \vect{y}}.
\end{equation*}
Mentioning Lipschitz continuity in the remainder of this article, we will always refer to $\norm[1]{\fcdot}$ if not stated otherwise.
In addition, we assume to know a function $\mu_f: \mathcal{P}(\domain) \mapsto \reals_+$, measuring the Lebesgue integral of $f$ on any connected sub-domain of $\domain$.
In the one-dimensional case, such a function is typically the anti-derivative of~$f$. 

Note that, $f$ represents a constraint function $f_j$ from~\labelcref{prob:minlp}.
In this section, we exclusively treat its approximation from below, but remark that the approximation from above can be treated analogously by considering $-f$ instead.

For the approximation, we assume some guarantee $\eps > 0$ to be given and aim to determine paraboloids $p^l(\vx)$ for $l \in [K]$ such that
\begin{equation}
	\label{eq:condition-1-ndim}
	\max_{l \in [K]} p^l(\vx) \geq f(\vx) - \eps, \tag{C1}
\end{equation}
and
\begin{equation}
	\label{eq:condition-2-ndim}
	\max_{l \in [K]} p^l (\vx) \leq f(\vx), \tag{C2}
\end{equation}
for all $\vx \in \domain$.
Condition~\labelcref{eq:condition-1-ndim} ensures that $(p^l)_l$ serve in their entity as an approximation of the desired guarantee, 
whereas~\labelcref{eq:condition-2-ndim} secures the one-sided property.
In fact, one could reformulate it to $p^l(\vx) \leq f(\vx)$ for all $l \in [K]$.

\subsection{MIP Model}
In the following model, we interpret the coefficients of the paraboloids as variables.
If solved with an objective value of zero, the corresponding solution then determines paraboloids that satisfy conditions~\eqref{eq:condition-1-ndim} and \eqref{eq:condition-2-ndim}.

Before stating the model explicitly, we need to introduce the parameters ${\delta \in (0, \varepsilon)}$ and $\nu \in (0, \delta/\varepsilon)$.
Informally, their choices influence the centering of the parabolic approximation inside the $\eps$-tube between $f$ and $f - \eps$. 
This becomes clearer throughout this section.
In addition, a constant $C \geq L$ needs to be chosen which bounds the maximal absolute slope of each paraboloid and thus controls their ``spiking''.

With these parameters at hand, we define the widths $\Delta t_i$ and $\Delta d_i$ for $i \in [n]$ such that 
\begin{equation}
	\label{eq:choice-delta-ti}
	\sum_{i = 1}^n \Delta t_i \leq \frac{n+1}{n} \frac{\varepsilon - \delta}{3L},
\end{equation}
and 
\begin{equation}
	\label{eq:choice-delta-di}
	\Delta d_i \leq \frac{2\nu\eps}{(\sqrt{3} - 1)n (C + L)}.
\end{equation}
Because the right-hand side of the above inequalities is strictly positive, such a choice is always possible. 
Without loss of generality, we can assume $\Delta t_i$ and $\Delta d_i$ to be chosen such that $(b_i - a_i)/\Delta t_i \in \naturals$ and $(b_i - a_i)/\Delta d_i \in \N$ for all $i \in [n]$. 
This allows to define two grids on the domain $\domain$ as
\begin{equation*}
	\mathcal{G}_\eps \coloneqq \bigtimes_{i = 1}^n \Defset[\Bigm|]{a_i + k \Delta t_i}{k = 0,1,\dots, \frac{b_i - a_i}{\Delta t_i}},
\end{equation*}
and
\begin{equation*}
	\mathcal{G} \coloneqq \bigtimes_{i = 1}^n \Defset[\Bigm|]{a_i + k\Delta d_i}{k=0,1,\dots,\frac{b_i - a_i}{\Delta d_i}}.
\end{equation*}
Constraints defined on these grids enforce the conditions~\eqref{eq:condition-1-ndim} and \eqref{eq:condition-2-ndim}, respectively.
If we collect the individual grid widths as the vectors ${\Delta \vt = (\Delta t_1, \dots, \Delta t_n)\T}$ and $\Delta \vd = (\Delta d_1,\dots,\Delta d_n)\T$, we are able to rewrite the domain $\domain$ as 
\begin{equation}
	\label{eq:domain-partition}
	\domain = \bigcup_{\vt \in \mathcal{G}_\eps \cap [\va, \vb)} [\vt, \vt + \Delta \vt] = 
	\bigcup_{\vd \in \mathcal{G} \cap [\va, \vb)} [\vd, \vd + \Delta\vd].
\end{equation}
Note that the arguments of each union intersect at most on their boundaries. 
For clear notation in the model, we abbreviate $\mathcal{B}(\vd) = [\vd, \vd + \Delta \vd]$.
Further, we denote all neighboring points to $\vt$ in $\mathcal{G}_\eps$, \ie, all points which differ in each coordinate by exactly $\Delta t_i$, as 
\begin{equation*} 
	\mathcal{N}(\vt) = \Defset[\Bigm|]{\vt +  \sum_{i = 1}^n u_i \Delta t_i \vect{e}_i \in \mathcal{G}_\eps}{	\vect{u} \in \{-1, 1\}^n}, 
\end{equation*}
where $\vect{e}_i$ denotes the $i$th unit vector.

In terms of variables, we denote the paraboloid coefficients $\alpha_i^l,\, \beta_i^l,\, \gamma^l$ for $i \in [n]$ and $l \in [K]$, 
specifying the $l$th paraboloid as $p^l(\vx) = \sum_{i = 1}^n \alpha_i^l x_i^2 + \sum_{i = 1}^n \beta_i^l x_i + \gamma^l$.
Furthermore, we introduce binary variables $s_\vt^l \in \{0, 1\}$ for all $l \in [K]$ and all grid points $\vt \in \mathcal{G}_\eps$.
These indicate whether the $l$th paraboloid fulfills $p^l(\vt) \geq f(\vt) - \delta$, which is a slight variation to~\labelcref{eq:condition-1-ndim} and ensures the latter.
Condition~\eqref{eq:condition-2-ndim}, however, is controlled by continuous variables $v_\vd^l \geq 0$ for $l \in [K]$ and $\vd \in \mathcal{G}$, which track violations of the integral between $f$ and a paraboloid $p^l$.

Now, we can finally formulate the \ac{mip}.
For a comprehensive presentation, nonlinear expressions are kept and their equivalent linear formulation is discussed afterwards.
\begin{subequations}
	\label{problem:original-multi-dim}
	\begin{align}
		\min \ && \sum_{l \in [K]} \sum_{\vect{d} \in \mathcal{G} \cap [\va, \vb)} v_\vd^l \label{obj:multi-dim}\\
		\text{s.t.} \ && p^l(\vt)  &\geq  f(\vt) - \delta - M_1(1 - s_\vt^l),& l \in [K],\, \vt \in \mathcal{G}_\eps, \label{eq:below-approx-multi-dim-1}\\
		&& \sum_{l \in [K]} s_\vt^l &\geq 1, & l \in [K],\, \vt \in \mathcal{G}_\eps, \label{eq:below-approx-multi-dim-2} \\
		&& \Abs{\frac{\mathrm{d}}{\mathrm{d}x_i} p^l(\vt')} & \leq 2L + M_2(1 - s^l_{\vt}),
		 \begin{split} l \in [K],\, i\in[n], \\
			\vt \in \mathcal{G}_\eps,\, \vt' \in \mathcal{N}(\vt),  \end{split} \label{eq:below-approx-multi-dim-3} \\
		&& p^l(\vd) & \leq f(\vd) - \nu\eps, & l \in [K], \, \vd \in \mathcal{G}, \label{eq:above-approx-multi-dim} \\
		&& v_\vd^l &\geq \int\limits_{\mathcal{B}(\vd)} p^l(\vx) - (f(\vx) - \nu\eps) \mathrm{d}\vx, & l \in [K],\, \vd \in \mathcal{G}, \label{eq:integral-tracking-multi-dim} \\
		&& \Abs{\frac{\mathrm{d}}{\mathrm{d}x_i} p^l(\va)} &\leq C, &l \in [K],\, i \in [n], \label{eq:gradient-bound-a} \\
		&& \Abs{\frac{\mathrm{d}}{\mathrm{d}x_i} p^l(\vb)} &\leq C, &l \in [K],\, i \in [n], \label{eq:gradient-bound-b}\\
		&& \alpha_i^l,\, \beta_i^l,\, \gamma^l &\in \reals, & l \in [K],\, i \in [n], \label{vars:paraboloid}\\
		&& s_\vt^l &\in \{0, 1\}, & l \in [K], \, \vt \in \mathcal{G}_\eps, \label{vars:contain-multi-dim}\\
		&& v_\vd^l & \geq 0, & l \in [K],\, \vd \in \mathcal{G}.	\label{vars:integral-tracking-multi-dim}
	\end{align}
\end{subequations}
Since $\frac{\mathrm{d}}{\mathrm{d}x_i} p^l(\vx) = 2\alpha_i^l x_i + \beta_i^l$ for $l \in [K]$, 
the derivative from above is a linear expression in the variables $\alpha_i^l$ and $\beta_i^l$.
Further, for a function $h(\vx) : \mathcal{D} \mapsto \reals$, we rewrite $\abs{h(\vx)} \leq C$ as $h(\vx) \leq C$ and $h(\vx) \geq -C$.
This allows to rewrite constraints~\eqref{eq:below-approx-multi-dim-3}, \eqref{eq:gradient-bound-a}, and \eqref{eq:gradient-bound-b} as linear inequalities.
Similarly, the anti-derivative of $p^l$ is a linear term in the variables $\alpha_i^l$, $\beta_i^l$ and $\gamma^l$. 
Combined with $\mu_f$ the integral in~\eqref{eq:integral-tracking-multi-dim} is evaluated and the constraints reduce to another set of linear inequalities.
This concludes the formulation of the problem by means of \ac{mip} techniques.

\subsection{Existence \& Validity of a Solution}

We have to show two aspects regarding model~\eqref{problem:original-multi-dim}: the existence of a solution for appropriate choice of parameters and the validity of conditions~\eqref{eq:condition-1-ndim} and \eqref{eq:condition-2-ndim} for such a solution.
The following theorem provides the first one.

\begin{theorem}[Existence]
	\label{thm:upper-bound-K}
	Let $0 < \Delta \leq \min\{ \frac{n+1}{n^2}\frac{\eps-\delta}{3L}, \frac{2\nu\eps}{(\sqrt{3} - 1)n (C + L)}, \frac{2\delta}{nL} \}$.
	Further, assume that $\Delta d_i = \Delta t_i = \Delta$ for all $i \in [n]$, $C = 2L\norm[\infty]{\vb - \va} / \Delta$, and $\nu = \delta/(2\eps)$.
	Then, for 
	\begin{equation*}
		K = \abs{\mathcal{G}_\eps} = \prod_{i = 1}^n \left\lceil \frac{b_i - a_i}{\Delta} \right\rceil,
	\end{equation*}
	problem~\labelcref{problem:original-multi-dim} has an optimal solution with an objective value of zero.
\end{theorem}

\begin{proof}
	A detailed proof can be found in~\Cref{subsec:proof-thm-existence}.
\end{proof}

Before turning to the validity of a solution, we demonstrate a connection between bounds to the partial derivatives and the Lipschitz continuity of a paraboloid. 

\begin{lemma}
	\label{le:p-Lipschitz}
	Let $p(\vx) = \sum_{i = 1}^n \alpha_i x_i^2 + \sum_{i = 1}^n \beta_i x_i + \gamma$ be a paraboloid and $\domain' = [\va', \vb']$.
	If there exists $C > 0$ such that 
	\begin{equation*}
		\Abs{\frac{\mathrm{d}}{\mathrm{d}x_i} p(\va')} \leq C \qquad \land \qquad \Abs{\frac{\mathrm{d}}{\mathrm{d}x_i} p(\vb')} \leq C,
	\end{equation*}
	it follows that
	\begin{equation*}
		\forall \vx \in \domain': \Abs{\frac{\mathrm{d}}{\mathrm{d}x_i} p(\vx)} \leq C,
	\end{equation*}
	for all $i \in [n]$.
	Further, $p$ is Lipschitz continuous on $\mathcal{D}'$ with Lipschitz constant $nC$ with respect to $\norm[\infty]{\fcdot}$ and with Lipschitz constant $C$ with respect to $\norm[1]{\fcdot}$.
\end{lemma}

\begin{proof}
	A detailed proof can be found in~\Cref{subsec:proof-le-p-Lipschitz}.
\end{proof}

The lemma shows that the bounds in~\labelcref{eq:gradient-bound-a} and~\labelcref{eq:gradient-bound-b} directly force a paraboloid defined by a feasible solution to be Lipschitz continuous with constant $C$.
Additionally, inequalities~\labelcref{eq:below-approx-multi-dim-3} ensure that a paraboloid $p^l$ with $s^l_\vt = 1$ is Lipschitz continuous on the neighborhood $\mathcal{N}(\vt)$ of $\vt$ with constant $2L$. 

Combining this implied Lipschitz continuity of a paraboloid with the Lipschitz continuity of $f$ turns the difference of both into a Lipschitz function itself.
We leverage this effect to prove bounds on this difference, while abstracting the situation.
We begin by establishing a lower bound based on the Lipschitz property and non-negativity on the vertices of a given domain. 

\begin{lemma}
	\label{le:g-lower-bound}
	Let $g \in \mathcal{C}^0(\domain')$ be Lipschitz continuous with constant $L_g > 0$ and $\domain' = [\va', \vb'] \subseteq \domain$.
	If $g(\vv) \geq 0$ for all vertices $\vv$ of $\domain'$, then it holds true that
	\begin{equation*}
		\min_{\vx \in \domain'} g(\vx) \geq -\frac{L_g n}{n+1} \norm[1]{\vb' - \va'}.
	\end{equation*}
\end{lemma}

\begin{proof}
	A detailed proof can be found in~\Cref{subsec:proof-le-g-lower-bound}.
\end{proof}

Note that this result provides an analogous upper bound to $\max_{\vx \in \domain'} g(\vx)$ when assuming $g(\vv) \leq 0$, respectively.
This would already suffice to construct an \ac{mipm} which provably returns a solution that satisfies~\labelcref{eq:condition-1-ndim} and~\labelcref{eq:condition-2-ndim} if the objective value is zero.
However, we have included the availability of a Lebesgue measure to be able to evaluate the integral in~\labelcref{eq:integral-tracking-multi-dim}, which can be leveraged to improve such bounds. 
In the following lemma, this is demonstrated.

\begin{lemma}
	\label{le:g-upper-bound}
	Let $g \in \mathcal{C}^0(\domain')$ be Lipschitz continuous with constant $L_g > 0$ and $\domain' = [\va', \vb'] \subseteq \domain$ full dimensional.
	If $g(\vv) \leq 0$ for all vertices $\vv$ of $\domain'$ and $\int_{\domain'} g(\vx) \mathrm{d}\vx \leq 0$, 
	then it holds true that 
	\begin{equation*}
		\max_{\vx \in \domain'} g(\vx) \leq \frac{\sqrt{3} - 1}{2} \Delta_{\max} n L_g,
	\end{equation*}
	where $\Delta_{\max} = \max_{i \in [n]} b_i' - a_i'$.
	
\end{lemma}

\begin{proof}
	A detailed proof can be found in~\Cref{subsec:proof-le-g-upper-bound}.
\end{proof}

To get a feeling for the magnitudes of both bounds, let's consider an approximately cubic domain, in particular, roughly equidistant along each coordinate such that $\norm[1]{\vb' - \va'} \approx n\Delta_{\max}$.
Then, it remains to notice that $0.366 \approx (\sqrt{3} - 1)/2 <  n/(n+1)$ for all $n \in \naturals$.
We remark that the bound is tighter in any case for $n = 1$ and in cases of approximately cubic domains also for $n \geq 2$. 

Now, \Cref{le:g-lower-bound} and~\Cref{le:g-upper-bound} are leveraged in combination with~\Cref{le:p-Lipschitz} to show that a solution of~\labelcref{problem:original-multi-dim} with zero objective satisfies~\labelcref{eq:condition-1-ndim} and~\labelcref{eq:condition-2-ndim}, respectively. 
We start with the former case.
	
\begin{theorem}[Validity of~\eqref{eq:condition-1-ndim}]
	\label{thm:correctness-1}
	If $(\alpha_i^l, \beta_i^l, \gamma^l)$, $s^l_\vt$, for $\vt \in \mathcal{G}_\eps$, $l \in [K]$, $i \in [n]$, satisfy the constraints~\labelcref{eq:below-approx-multi-dim-1,eq:below-approx-multi-dim-2,eq:below-approx-multi-dim-3}, 
	then the paraboloids $(p^l)_{l\in K}$ defined by the parameters $(\alpha_i^l, \beta_i^l, \gamma^l)$ fulfill~\eqref{eq:condition-1-ndim}.
\end{theorem}

\begin{proof}
	Recall from~\labelcref{eq:domain-partition} that $\domain$ is the union of $[\vt, \vt + \Delta \vt]$ over all $\vt \in \mathcal{G}_\eps \cap [\va, \vb)$.
	Hence, consider an arbitrary but fixed $\vt \in \mathcal{G}_\eps \cap [\va, \vb)$ and the box $[\vt, \vt + \Delta \vt]$.
	
	Observe that all vertices $\vt'$ of $[\vt, \vt + \Delta\vt]$ are again vectors in $\mathcal{G}_\eps$.
	Therefore, \labelcref{eq:below-approx-multi-dim-2} with the integrality restriction on $s_\vt^l$ ensure that for all such $\vt'$, 
	there exists $l \in [K]$ such that $s_{\vt'}^l = 1$.
	We collect these indices in 
	\begin{equation*} 
		I_\vt = \{ l \in [K] \mid \exists \vt' \text{ vertex of } [\vt, \vt + \Delta \vt]: s_{\vt'}^l = 1\}.
	\end{equation*}
	Now, for all $l \in I_\vt$, the inequalities~\labelcref{eq:below-approx-multi-dim-3} in combination with~\Cref{le:p-Lipschitz} ensure that paraboloid $p^l$ is Lipschitz continuous on $[\vt, \vt + \Delta\vt]$ with constant $2L$.
	Indeed, we can show that this Lipschitz property transfers to $\max_{l \in I_\vt} p^l(\vx)$.
	To this end, consider $\vx,\, \vect{y} \in [\vt, \vt + \Delta\vt]$ and let $l_\vx \in \argmax_{l \in I_\vt} p^l(\vx)$, as well as $l_\vect{y} \in \argmax_{l \in I_\vt}p^l(\vect{y})$. 
	If $p^{l_\vx}(\vx) \geq p^{l_\vect{y}}(\vect{y})$, we can derive
	\begin{align*}
		\abs{\max_{l \in I_\vt} p^l (\vx) - \max_{l \in I_\vt}p^l (\vect{y})} &= \abs{p^{l_\vx}(\vx) - p^{l_\vect{y}}(\vect{y})} = p^{l_\vx}(\vx) - p^{l_\vect{y}}(\vect{y}) \\
		& \leq p^{l_\vx}(\vx) - p^{l_\vx}(\vect{y}) \leq \abs{p^{l_\vx}(\vx) - p^{l_\vx}(\vect{y})} \leq 2L\norm[1]{\vx - \vect{y}},
	\end{align*}
	where we used the fact $p^{l_\vect{y}}(\vect{y}) \geq p^l(\vect{y})$ for all $l \in I_\vt$.
	The case $p^{l_\vx}(\vx) \leq p^{l_\vect{y}}(\vect{y})$ follows analogously. 
	Now, define $g(\vx) \coloneqq \max_{l \in I_\vt}p^l(\vx) - (f(\vx) - \delta)$.
	The above result and the Lipschitz property of $f$ give by the triangle inequality that $g$ is Lipschitz continuous with constant $L_g = 2L + L = 3L$.
	
	In addition, for all vertices $\vt'$ of $[\vt, \vt + \Delta \vt]$, there exists $l \in I_\vt$ with $s_{\vt'}^l = 1$ and by~\labelcref{eq:below-approx-multi-dim-1} ${p^l(\vt') \geq f(\vt') - \delta}$.
	In particular, $\max_{l \in I_\vt} p^l(\vt') \geq f(\vt') - \delta$ and thus $g(\vt') \geq 0$.
	This allows to apply \Cref{le:g-lower-bound} and we receive
	\begin{equation*}
		\min_{\vx' \in [\vt, \vt + \Delta\vt]} g(\vx') \geq - \frac{3Ln}{n+1} \norm[1]{\Delta \vt}.
	\end{equation*}
	Due to~\eqref{eq:choice-delta-ti} it holds true that $\norm[1]{\Delta\vt} = \sum_{i = 1}^n \Delta t_i \leq \frac{n+1}{n} \frac{\eps - \delta}{3L}$.
	This then gives that for all $\vx \in [\vt, \vt+ \Delta\vt]$, it is
	\begin{equation*}
		\max_{l \in [K]}p^l(\vx) - (f(\vx) - \delta) \geq  g(\vx) \geq \min_{\vx' \in [\vt, \vt + \Delta\vt]} g(\vx') \geq \delta - \varepsilon,
	\end{equation*}
	which is equivalent to 
	\begin{equation*}
		\max_{l \in [K]}p^l(\vx) \geq f(\vx) - \varepsilon.
	\end{equation*}
	Since $\vt$ was considered arbitrary and the domain $\domain$ is a union of all such boxes $[\vt, \vt + \Delta \vt]$, the claim follows.
\end{proof}
	
One can observe that the first condition~\labelcref{eq:condition-1-ndim} is enforced by means of~\labelcref{eq:below-approx-multi-dim-1,eq:below-approx-multi-dim-2,eq:below-approx-multi-dim-3} together with the variables $s_\vt^l$ and their integrality restriction.
For~\labelcref{eq:condition-2-ndim}, the remaining constraints together with the non-negative variables $v_\vd^l$ and the objective~\labelcref{obj:multi-dim} come into play.

\begin{theorem}[Validity of~\eqref{eq:condition-2-ndim}]
	\label{thm:correctness-2}
	If $(\alpha_i^l, \beta_i^l, \gamma^l)$, $v^l_\vd$, for $\vd \in \mathcal{G}$, $l \in [K]$, $i \in [n]$, satisfy the constraints~\labelcref{eq:above-approx-multi-dim,eq:integral-tracking-multi-dim,eq:gradient-bound-a,eq:gradient-bound-b} and have objective value~\labelcref{obj:multi-dim} of zero,
	then the paraboloids $(p^l)_{l\in [K]}$ defined by the parameters $(\alpha_i^l, \beta_i^l, \gamma^l)$ fulfill  \eqref{eq:condition-2-ndim}.
\end{theorem}

\begin{proof}
	Recall from~\labelcref{eq:domain-partition} that $\domain$ is the union of $[\vd, \vd + \Delta \vd]$ over all $\vd \in \mathcal{G} \cap [\va, \vb)$.
	Hence, consider an arbitrary but fixed $\vd \in \mathcal{G} \cap [\va, \vb)$ and the box $[\vd, \vd + \Delta \vd]$.
	
	Observe that all vertices $\vv$ of this box are again vectors in $\mathcal{G}$. 
	Therefore, constraints~\eqref{eq:above-approx-multi-dim} ensure that for all vertices $\vv$ and $l \in [K]$, it holds true that $p^l(\vv) \leq f(\vv) - \nu\eps$.
	
	Now, consider a fixed $l \in [K]$ and let $g(\vx) \coloneqq p^l(\vx) - (f(\vx) - \nu\eps)$.
	Above gives $g(\vv) \leq 0$ for all vertices $\vv$ of $[\vd, \vd+\Delta\vd]$.
	In addition, by~\Cref{le:p-Lipschitz} and the constraints~\labelcref{eq:gradient-bound-a,eq:gradient-bound-b}, $p^l$ is Lipschitz with constant $C$.
	Using the triangle inequality we can show that $g$ is Lipschitz continuous with constant $ L_g = C+L$.
	By assumption the objective~\labelcref{obj:multi-dim} is zero and thus $v_\vd^l = 0$.
	Combined with~\labelcref{eq:integral-tracking-multi-dim} this implies $0 \geq \int_{[\vd, \vd+\Delta\vd]} g(\vx) \mathrm{d}\vx$.
	Therefore, we can apply \Cref{le:g-upper-bound} and receive 
	\begin{equation*}
		p^l(\vx) - (f(\vx) - \nu\eps) = g(\vx) \leq \max_{\vx' \in [\vd, \vd+\Delta\vd]} g(\vx') \leq \frac{\sqrt{3} - 1}{2} n (C + L) \max_{i \in [n]}\Delta d_i ,
	\end{equation*}
	for all $\vx \in [\vd, \vd + \Delta\vd]$.
	From choosing $\Delta d_i$ according to~\labelcref{eq:choice-delta-di}, we recall that $\Delta d_i \leq (2\nu\eps)/((\sqrt{3} - 1) n (C  + L))$ and can conclude
	\begin{equation*}
		p^l(\vx) \leq f(\vx).
	\end{equation*}
	Since $l \in [K]$  and $\vd$ were considered arbitrary and the domain $\domain$ is a union of all such boxes $[\vd, \vd + \Delta \vd]$, the claim follows.
\end{proof}
	
We have now established a model which is suitable to find an approximation with the desired conditions~\labelcref{eq:condition-1-ndim,eq:condition-2-ndim}.

\section{Approximation and Relaxation by Paraboloids}
\label{sec:gbm}

The previous \ac{mipm} model is embedded in two search strategies to determine a small number of paraboloids for approximating a given function on a certain domain. 
Afterwards, we demonstrate ways to incorporate those approximations into \ac{minlp} problems and discuss potential benefits and drawbacks. 
If coupled with a relaxation of the general nonlinear constraints, we point out that the original \ac{minlp} is reduced to an \ac{miqcpm}.
In a complementary step, we motivate creating a lookup table of approximations, which can be accessed when tackling \ac{minlp} problems by the parabolic relaxation.

\subsection{Parabolic Approximation by Few Paraboloids}
\label{subsec:para-approx-algo}

We start by recalling the setting: 
Let $f$ be a Lipschitz continuous function with constant $L > 0$, which is to be approximated, and $\eps > 0$ the desired approximation guarantee.
Assume further that the choices of $\delta,\,\nu,\, C, \Delta t_i, \Delta d_i > 0$ comply with the requirements~\labelcref{eq:choice-delta-ti,eq:choice-delta-di}, as well as~\Cref{thm:upper-bound-K}.
The latter ensures the existence of a solution to problem~\eqref{problem:original-multi-dim} with $\overline{K} = \abs{\mathcal{G}_\eps}$, which meets the conditions~\eqref{eq:condition-1-ndim} and \eqref{eq:condition-2-ndim}.

Since each paraboloid results in an additional constraint when incorporated in the original problem later, 
we aim to find a small  $K^* < \overline{K}$.
A first idea is the application of a binary search in the discrete interval $[1, \overline{K}] \enifed [\underline{K}, \overline{K}]$.
To this end, one starts by solving problem~\eqref{problem:original-multi-dim} with $K = \lfloor (\overline{K}+ \underline{K})/2 \rfloor$. 
If an optimal solution with objective zero is returned, it defines a set of paraboloids which satisfy~\labelcref{eq:condition-1-ndim,eq:condition-2-ndim}.
Therefore, the upper bound is updated such that $\overline{K} \gets K$.
Otherwise, $\underline{K} \gets K + 1$.
After updating a bound, this procedure is restarted with respect to the new interval.
It terminates when $\overline{K} \leq \underline{K}$ and returns $K^* \gets \overline{K}$ as the minimal number of paraboloids.
A formalization can be found in~\Cref{alg:bin-search}.

However, since the initial $\overline{K}$ is typically large, we propose a slight adjustment to the classical binary search.
For a start, we set $K = 1$ and if the algorithm asserts in step~\labelcref{algl:bin-search-assertion}, $K$ is simply doubled in the next iteration, while respecting $\overline{K}$ as an upper bound.
This is performed, until the algorithm does not assert, \ie, until it detects a feasible number of paraboloids and updates the upper bound.
Then, \Cref{alg:bin-search} resumes as described and computes the minimal number of paraboloids $K^*$.

\begin{algorithm}
\caption{Number of Paraboloids -- Binary Search}
\label{alg:bin-search}

\begin{algorithmic}[1]
	\setcounter{ALC@unique}{0}
	\REQUIRE an upper bound of paraboloids $\overline{K}$
	\ENSURE the minimal number of paraboloids $K^*$ satisfying~\eqref{eq:condition-1-ndim} and~\eqref{eq:condition-2-ndim}
	\STATE Set $\underline{K} \gets 1$
	\WHILE{$\overline{K} > \underline{K}$}
		\STATE Set $K \gets \lfloor (\overline{K}+ \underline{K})/2 \rfloor$
		\STATE Solve~\eqref{problem:original-multi-dim} for $K$ $\to$ ``infeasible'' or objective $c^K$
		\IF{``infeasible'' \OR $c^K > 0$} \label{algl:bin-search-assertion}
		\STATE Set $\underline{K}\gets K + 1$
		\ELSE
		\STATE Set $\overline{K} \gets K$
		\ENDIF		
	\ENDWHILE
	\RETURN $K^* \gets \overline{K}$
\end{algorithmic}
\end{algorithm}

We note that the minimality of $K^*$ is considered with respect to the \ac{mip} formulation~\labelcref{problem:original-multi-dim}.
Indeed, there might exist an approximation with fewer paraboloids that satisfies the conditions~\eqref{eq:condition-1-ndim} and~\eqref{eq:condition-2-ndim}, 
but is infeasible for~\eqref{problem:original-multi-dim}.
For instance, two of the paraboloids may not overlap at a predefined discretization point.
In addition, dependent on the particular form of $f$, an analytical solution may be computable by means of analytical and numerical methods.

It is also noteworthy that the size of problem~\labelcref{problem:original-multi-dim} can increase to a non-negligible extend, especially when dealing with a large Lipschitz constant or a small accuracy. 
This effect is catalyzed when considering a large number of paraboloids $K$, as the entire problem size grows linearly in $K$.
To deal with a potentially large problem size, we propose a second approach. 
It relaxes the original problem~\labelcref{problem:original-multi-dim} by omitting constraints~\labelcref{eq:below-approx-multi-dim-3} 
and determines the discretization widths $\Delta t_i$ and $\Delta d_i$ adaptively, not considering the assumptions~\labelcref{eq:choice-delta-ti,eq:choice-delta-di}.
These measures decrease the problem size,
but a solution may lack the desired conditions~\eqref{eq:condition-1-ndim} and~\eqref{eq:condition-2-ndim}.
Hence, we incorporate checks by solving small \ac{nlp} problems $\min_\vx (f(\vx) - \eps - \max_l p^l(\vx))$ and $\min_\vx (\max_l p^l(\vx) - f(\vx))$ on the domain $\domain$ for a given $(p^l)_l$.
This is achieved by remodeling to $K$ inequalities or by solving $K$ simpler problems, respectively.
If the objectives are non-positive, the solution fulfills the desired property and we terminate. 
Otherwise, the number of paraboloids or the number of discretization points is increased and another loop is executed. 
We summarize this implementation in~\Cref{alg:practice-para}.

\begin{algorithm}
	\caption{Number of Paraboloids -- Practical Implementation}
	\label{alg:practice-para}
	\begin{algorithmic}[1]
		\setcounter{ALC@unique}{0}
		\REQUIRE an upper bound of paraboloids $\overline{K}$, start values of discretization points $T_0,\, D_0\in \N$ and paraboloids $K_0 < \overline{K}$ 
		\ENSURE a small number of paraboloids $K^*$ satisfying~\eqref{eq:condition-1-ndim} and~\eqref{eq:condition-2-ndim}
		\STATE Set $T \gets T_0$, $D \gets D_0$, $K \gets K_0$
		\WHILE{$K < \overline{K}$}
		\STATE Set $T$-bound $\gets$ \FALSE, $D$-bound $\gets$  \FALSE
		\STATE Compute $\Delta t_i$, $\Delta d_i$ according to $T, \, D$ for $i \in [n]$
		\STATE Solve~\eqref{problem:original-multi-dim} without~\eqref{eq:below-approx-multi-dim-3} $\to$ ``infeasible'' or solution $(p^l)_{l \in [K]}$ with objective $c^K$
		\IF{``infeasible'' \OR $c^K > 0$}
		\STATE Increment $K$
		\ELSE
		\STATE Check~\eqref{eq:condition-2-ndim} by computing $c^l \define \min\limits_{\vx \in \domain}\{p^l(\vx) - f(\vx) \}$ for $l \in [K]$ \label{algl:prac-check1}
		\IF{$\exists l \in [K]: c^l > 0$}
		\STATE Increase $T$
		\ELSE
		\STATE Set $T$-bound $\gets$  \TRUE
		\ENDIF
		\STATE Shift each $p^l$ by $-c^l$, $l \in [K]$
		\STATE Check~\eqref{eq:condition-1-ndim} by computing $c_\eps \define \min\limits_{\vx \in \domain, y} \{y - (f(\vx) - \eps) \text{ \st \ } y \geq p^l(\vx), \, l \in [K]\}$ \label{algl:prac-check2}
		\IF{$c_\eps < 0$}
		\STATE Increase $D$
		\ELSE
		\STATE Set $D$-bound $\gets$ \TRUE
		\ENDIF 
		\ENDIF
		\IF{$T$-bound \AND $D$-bound}
		\STATE Exit loop
		\ENDIF
		
		\ENDWHILE
		\RETURN $K^* \gets K$
	\end{algorithmic}
	
\end{algorithm}

\subsection{Parabolic Relaxation}
\label{subsec:para-relax}

Turning to the original \ac{minlp} problem~\labelcref{prob:minlp}, we now assume to have a respective approximation $(p^l_j)_{l \in [K^j]}$ for each constraint function~$f_j$.
Then, we replace each $f_j$ in~\eqref{prob:minlp} leading to
\begin{align*}
	\min \ & c(\vx) \\
	\st \ & \max_{l \in [K^j]} p^l_j(\vx)  \leq y_j, & j \in J,, \\
	& \vx \in \Omega.
\end{align*}
Since the maximum over all paraboloids is involved in a ``$\leq$''-constraint, we equivalently reformulate this problem to
\begin{subequations}
	\label{prob:direct-approach-reformulated}
	\begin{align}
		\min \ & c(\vx) \\
		\st \ & p^l_j(\vx)  \leq y_j, & j \in J, l \in [K^j], \\
		& \vx \in \Omega.
	\end{align}
\end{subequations}
To the best of our knowledge, the only similar approach to this is performed in~\cite{Mangasarian2006}, where the authors use the minimum of paraboloids to compute a non-convex underestimator of a non-convex objective. 

Since, without loss of generality, $c(\vx)$ can be assumed to be linear, this procedure creates a relaxation of the original \ac{minlpm} to an \ac{miqcpm}.
In the realm of \ac{minlp}, exact solutions are typically out of scope and one is satisfied with an approximate solution with respect to some $\eps > 0$. 
If a small enough approximation guarantee is considered when computing the approximations, solving the resulting \ac{miqcp} problem~\labelcref{prob:direct-approach-reformulated} returns a satisfactory solution for the original problem. 
This allows solvers tailored for \ac{miqcp} such as GloMIQO~\cite{Misener2013} to be applied to an even broader problem class.

When solving such MIQCP problems, we emphasize that state-of-the-art solvers typically lift the problem to a higher dimension, in which quadratic functions are reduced to linear terms and simple squares, \eg, $\tilde{x} = x^2$. 
After the introduction of the first paraboloid $p^1$ each additional one $p^l$ can therefore be considered as an additional linear constraint in the extended space.
This allows to control the complexity of our approach, since tighter paraboloid relaxations can be realized by adding linear constraints.  
In the light of this, we propose to keep the general nonlinear constraints and to incorporate the parabolic constraints in addition.

Note that the representation of an \ac{minlpm} like~\labelcref{prob:minlp} is explicitly tailored for either strategy, 
but does not constitute a major limitation. 
For instance, assume that the right-hand side $y_j$ is substituted by zero, resulting in a more general constraint ${f_j(\vx) \leq 0}$.
Then, two additional steps must be performed to apply the parabolic relaxation: 
First, one also requires an approximation of the hypograph of $f_j$, which can be obtained by executing the procedures in~\Cref{subsec:para-approx-algo} for $-f_j$. 
Second, one introduces an auxiliary variables $z_j$, which serves as a replacement for $f_j$, leaving the original constraint $z_j \leq 0$, and bounds $z_j$ from above and below by the computed approximations.

Now, going one step further, notice that most practical \ac{minlp} problems are \emph{factorable}. 
This allows to represent each constraint function in terms of an \emph{expression tree}, see, for instance,~\cite{Schichl2005}.
That is, each constraint function comprises of variables, constants, and nested elementary operations (such as '+', '$\cdot$' or one-dimensional univariate functions). 
This can be depicted as an acyclic graph where every node is such a variable, constant, or operation, and this graph is called expression tree.
By introduction of auxiliary variables a factorable \ac{minlp} problem can be transformed to only contain one-dimensional functions, as well as linear and bivariate terms. 
For an explanatory introduction of this procedure, see~\cite{Morsi2013}.

This representation is explicitly beneficial for the parabolic relaxation procedure because typically there already exists structures like $f_j(\vx) = y_j$, where bounding $y_j$ by paraboloids and relaxing this constraint is already sufficient.
Hence, when referring to either strategy (replacement and addition) on a factorable \ac{minlpm}, it is always assumed to apply this on suitable function components.

We like to highlight one potential pitfall for this procedure. 
The approximation guarantee for a one-dimensional component of an original constraint $f_j$ might not propagate directly.
The latter can worsen the overall accuracy and requires thorough analysis of the expression tree's depth.
We think of these aspects as accountable if respected.

\subsection{Lookup Tables}
\label{subsec:lookup-table}

To review our claim about factorability, we examine instances from the MINLPLib~\cite{MINLPLib}, a test set of various \ac{minlp} problems. 
Note that the majority of those instances is available in the OSIL-format, which is XML-based and mimics the expression tree.
Thus, these instances are factorable.

This format also facilitates the analysis of the components of an instance.
In \Cref{fig:1-d-occurence}, we display the number of occurrences of one dimensional functions for all such instances.
Hereby, occurrences of `power' are not distinguished whether the exponent or the base is a variable or a constant. 
Further, the square function is excluded, since we argue to not approximate but keep those terms.

\begin{figure}
\caption{Cumulative occurrences of one dimensional function types in MINLPLib instances in OSIL-format}
\label{fig:1-d-occurence}

\begin{tikzpicture}[scale=0.9]
	\begin{axis}[
		ybar,
		xtick=data,
		xlabel={Function},
		xlabel style={yshift=-10pt,xshift=-10pt},
		ylabel={Occurrences},
		ylabel style={yshift=-5pt},
		nodes near coords,
		nodes near coords align={vertical},
		bar width=12pt,
		width=14cm,
		height=8cm,
		enlarge x limits=0.05,
		ymin=0,
		ymax=92000,
		legend style={at={(0.5,-0.2)},anchor=north},
		x tick label style={rotate=45, anchor=east, align=center},
		xtick pos=bottom,
		ytick=\empty,
		symbolic x coords = {exp, cos, sin, sqrt, ln, erf, abs, power, signpower, log10, tanh, gammaFn, min}
		]
		\addplot[fill=gray!50] coordinates {(exp, 83854)(cos, 73034)(sin, 66917)(sqrt, 58583)(ln, 36504)(erf, 23165)(abs, 17611)(power, 11223)(signpower, 2426)(log10, 523)(tanh, 339)(gammaFn, 6)(min, 1)};
	\end{axis}
\end{tikzpicture}

\end{figure}
 
Observe that the exponential, the cosine, the sine, and the square root function occur very frequently.
This allows to use a parabolic approximation to these functions repeatedly when tackling an entire subset of these problems.
Naturally, the domain of their arguments varies dependent on the particular instance.
However, an approximation for a wide domain is always valid for a subdomain.
Therefore, we advocate to approximate these one dimensional functions with respect to some generic (\eg, the most common) domains and different approximation guarantees a-priori.
The result serves as a lookup table for the parabolic relaxation.

The creation and usage of such a table can be interpreted as the common paradigm in the field of Machine Learning (ML).
Once trained under potentially high computational effort, the ML approach is reused without significant cost in every similar task. 
This applies for the lookup table and tackling a class of similar optimization problems with parabolic approximations.

At this point, it is important to note that a concatenation of approximations, \ie, sets of paraboloids, of two overlapping domains might not result in a globally valid one-sided approximation for the united domain.
For instance, since the paraboloids are computed with respect to the original domain, one paraboloid can even intersect the function to approximate at a bound
and leads to a violation of property~\eqref{eq:condition-2-ndim} right outside the domain.

\begin{remark}
	\label{re:unbounded-approx}
	For the special case of the sine function, we highlight a potential enhancement of the procedure, which allows for concatenation. 
	At first, one can find a parabolic approximation to the sine function from $-\pi/2$ to $3\pi/2$ and hereby restrict the sign of the quadratic coefficients to be non-positive.
	Then, we can observe that the approximation fulfills~\eqref{eq:condition-2-ndim} on $\R$ entirely, as the paraboloids intersect somewhere below -1 at the bounds and decrease outside of the domain.
	At a second step, an approximation for an arbitrary (finite) domain can be computed by simply shifting the paraboloids appropriately and concatenating the solutions due to the periodicity of the sine.
	We even like to mention an approach by integer programming for unbounded domains for sine functions, where the integer variable keeps track of the period as an offset with respect to $[-\pi/2, 3\pi/2]$ and shifts the paraboloids accordingly.
	Naturally, both ideas can be applied analogously for the cosine function.
\end{remark}

\section{Computational Results}
\label{sec:computations}

We investigate the extend to which parabolic approximation and relaxation is sufficient for tackling \ac{minlp} problems in practice.
To this end, we have implemented the \ac{mip}-based approaches for one-sided parabolic approximations discussed in~\Cref{subsec:para-approx-algo} and analyze their ability to approximate functions (such as frequently occurring ones, see~\Cref{fig:1-d-occurence}) on different generic domains with varying accuracies.
These results are used as a lookup table (compare~\Cref{subsec:lookup-table}) to apply the parabolic relaxation in two variants described in~\Cref{subsec:para-relax} to MINLPLib~\cite{MINLPLib} instances.

All computational experiments are conducted on single nodes with Intel Xeon Gold 6326 ``Ice Lake'' multicore processors with 2.9GHz per core.
The accessible RAM is set to \SI{32}{\giga\byte} for the parabolic approximation and \SI{64}{\giga\byte} for the parabolic relaxation with four and eight cores, respectively.
All implementations are based on \software{Python 3.11.7}.
The optimization problems are modeled via \software{Pyomo}~\cite{Hart2011pyomo,Bynum2021pyomo}.
For the parabolic approximation, we solve the \ac{mip} models with \software{Gurobi~11.0.3}~\cite{Gurobi11} as a direct solver within \software{Pyomo} and leverage \software{SCIP~8.1} as part of \software{GAMS} in version 46.4.0 for the checks. 
For the parabolic relaxation, we solve the resulting problems with \software{Gurobi~11.0.1} and \software{SCIP~8.1}, both from inside the \software{GAMS} framework.
If not stated otherwise, default settings are used.
All code is available at \url{https://github.com/adriangoess/paraboloids.git} under a MIT license.
The performance profiles for evaluation in~\Cref{subsec:comp-para-relax} are created with the software~\cite{Siqueira2016perprof}.

\subsection{Parabolic Approximation}
\label{subsec:comp-para-approx}

To demonstrate the ability of our methods to approximate Lipschitz functions, we randomly generate a zigzag function as the first element for a test set. 
It is constructed by sampling uniformly and we refer to~\Cref{subsec:zigzag-def} for details and the bold drawing in~\Cref{fig:zigzag-approx-below} for an intuition.
Further, since we aim to re-use the approximations computed here, we refer to~\Cref{fig:1-d-occurence} and claim that an appropriate test set of functions is supposed to contain the exponential and the sine function. 
Approximations for $\ln$ are assumed to require similar effort and for $\cos$ they can be derived by scaling and translation.
In order to additionally investigate the possibility of approximating higher order polynomials by parabolic (\ie, quadratic) terms, we include $x^3$ into the test set.

Since we interpret the outcome discussed in this section as a lookup table, we take generic domains for each function into account. 
For zigzag, $\exp$, and $x^3$, we thus consider an interval of $[-5, 5]$ and several subintervals. 
For $\sin$, we take the two periods $[-\pi/2, 3\pi/2]$ and $[0, 2\pi]$, as well as half of those into account, see~\Cref{tab:approx-domains} for an overview.
Each function is approximated from above and below in order to highlight potential differences in the necessary number of paraboloids.
For the approximation accuracy, we consider $\eps \in \{10^0, 10^{-1}, 10^{-2}, 10^{-3}\}$.

\begin{table}
	\centering
	\begin{tabular}{lcccccc}
		\toprule
		& \multicolumn{6}{c}{domain $\domain$} \\ \cline{2-7}
		\\[-1em]
		zigzag & $[-5, -1]$ & $[-2,\, 2]$ & $[\,\phantom{-}1, \ \ 5]$ & $[-5, 0]$ & $[0, -5]$ & $[-5, \phantom{\pi}5]$ \\
		$\exp$ &$[-5, -2]$ & $[-2,\, 2]$ & $[\, -5,\ \ 2]$& $[\phantom{-}2, 5]$& $[-2, 5]$ & $[-5, \phantom{\pi}5]$  \\ 
		$x^3$ & $[-2, \phantom{-}2]$ & $[-5, \, 0]$ & $[\, \phantom{-}0,\ \ 5]$ & $[-5, 2]$ & $[-2, 5]$ & $[-5,\phantom{\pi} 5]$ \\ 
		$\sin$  & $[-\frac{\pi}{2}, \ \frac{\pi}{2}]$ & $[\,\frac{\pi}{2}, \frac{3\pi}{2}]$ &$[-\frac{\pi}{2}, \frac{3\pi}{2}]$ & $[\phantom{-}0, \pi]$ & $[\pi,\, 2\pi]$ & $[\phantom{-}0, 2\pi]$ \\
		\bottomrule
	\end{tabular}
	\caption{Functions and domains to approximate by paraboloids.}
	\label{tab:approx-domains}
\end{table}

For the zigzag function, we bound the interval for a uniform sample such that its Lipschitz constant is upper bounded by 1. 
This can always be enforced by scaling. 
Note that the Lipschitz constants for $\exp$ and $x^3$ can be computed as the maximal absolute derivative at the bounds. 
For instance, on $[-5, 2]$, this is $\exp(2)$ and $2(-5)^2$. 
For $\sin$, we just consider the Lipschitz constant of 1, as this is the globally largest absolute derivative.

Recall that two methods for parabolic approximation have been presented in~\Cref{subsec:para-approx-algo}.
One (\texttt{exact-MIP}) considers problem~\labelcref{problem:original-multi-dim} unchanged and performs a binary search on the number of paraboloids, compare~\Cref{alg:bin-search}.
The other (\texttt{practical-MIP}) relaxes problem~\labelcref{problem:original-multi-dim} to facilitate the solving process, but incorporates manual checks of conditions~\labelcref{eq:condition-1-ndim} and~\labelcref{eq:condition-2-ndim}, see~\Cref{alg:practice-para} steps~\labelcref{algl:prac-check1,algl:prac-check2}.
These checks involve solving very small \ac{nlp} problems, which we perform via \software{SCIP}.

Both methods require the specification of parameters to control the problem size.
The method \texttt{exact-MIP} requires the definition of $\delta \in (0, \eps)$ and $\nu \in (0, \delta/\eps)$, which we set to $\eps/2$ and $\delta/(2\eps)$, respectively.
Instead, for \texttt{practical-MIP}, the initial numbers of discretization points $T_0$ and $D_0$ need specification.
Informally, the greater the Lipschitz constant, the wider the domain, or the smaller the accuracy, the more discretization points are required to satisfy the conditions. 
Translating these relations in a fractional term gives $T_0, \, D_0 \approx L\abs{\domain}/\eps$.
To keep the initial solution times low, we additionally divide by ten, resulting in $T_0 = D_0 = \lceil L\abs{\domain} / (10\eps) \rceil$.
For more details, we refer to the repository.

Based on these parameters, each method poses \ac{mip} problems which are solved using \software{Gurobi}.
To maintain memory for the modeling, we set a limit of \SI{24}{\giga\byte} from the available \SI{32}{\giga\byte} for the solving process.
In addition, the usable threads are limited to four and the solving time to \SI{3600}{\second}.
Together with an iteration limit of 24, each run is theoretically bounded in time by \SI{24}{\hour}, but typically falls short of it.
We recall that when determining an initial feasible upper bound of paraboloids, the search number is doubled in every unsuccessful iteration. 
This can lead to a problem size that causes memory-based failures already in the modeling stage. 
To avoid such, we terminate whenever a problem shows more than 1,000 binary and 200,000 continuous variables and proceed with another iteration of the search respecting a smaller number.

Diving into the results, we start with the approximations of the zigzag function from below, see~\Cref{tab:number-paraboloids-zigzag-below}.
Although having the theoretical guarantees as described in~\Cref{sec:approx}, the size of the corresponding \ac{mip} problems in \texttt{exact-MIP} for decreasing accuracy appears too large to be computationally tractable. 
Assuming the number of paraboloids to be fixed, we note that when decreasing the accuracy by a factor of 10, in problem~\labelcref{problem:original-multi-dim} the number of constraints grows by a factor of approximately 100 and the number of variables by 10. 
However, decreasing the accuracy by a certain factor intuitively results in a number of paraboloids increasing by a similar order of magnitude, which in turn enlarges the resulting problem even further. 
The adaption made for \texttt{practical-MIP} seems to cope with this increase to some extent, as it -- in contrast -- returns a result for the largest domain $[-5, 5]$ and $\eps=10^{-1}$ or for $[-5, -1]$ and $\eps = 10^{-2}$.
In~\Cref{fig:zigzag-approx-below} we visualize the former result.

Interestingly, the number of paraboloids for the approximation is smaller with \texttt{practical-MIP}.
This shows that the combination of constraints in problem~\labelcref{problem:original-multi-dim} represents a sufficient condition for a global one-sided approximation, but not a necessary one.
Since the results for the approximation from above are quite similar, we skip their explicit interpretation and refer to~\Cref{tab:number-paraboloids-zigzag-above} in~\Cref{subsec:approx-zigzag}.

\begin{table}[h]
	\centering
	\begin{tabular}{ccccccccccc}
		\toprule
		& & \multicolumn{9}{c}{$\eps$} \\ \cline{3-11}
		& & \multicolumn{9}{c}{\vspace{-0.28cm}} \\
		$\domain$ & & $10^{0}$ & $10^{-1}$ & $10^{-2}$ & $10^{-3}$  & & $10^{0}$ & $10^{-1}$ & $10^{-2}$ & $10^{-3}$ \\ \cline{3-6}\cline{8-11}
		\\[-1em] 
		$[-5, -1]$ &  & 2 & 5 & - & - &  & 1 & 4 & 41 & - \\ 
		\\[-1em] 
		$[-2, \phantom{-}2]$ &  & 2 & 9 & - & - &  & 1 & 4 & - & - \\ 
		\\[-1em] 
		$[\phantom{-}1, \phantom{-}5]$ &  & 2 & 7 & - & - &  & 1 & 4 & - & - \\ 
		\\[-1em] 
		$[-5, \phantom{-}0]$ &  & 2 & 8 & - & - &  & 1 & 5 & - & - \\ 
		\\[-1em] 
		$[\phantom{-}0, \phantom{-}5]$ &  & 2 & 8 & - & - &  & 1 & 4 & - & - \\ 
		\\[-1em] 
		$[-5, \phantom{-}5]$ &  & 4 & - & - & - &  & 2 & 9 & - & - \\ 
		& & \multicolumn{4}{c}{\texttt{exact-MIP}} & & \multicolumn{4}{c}{\texttt{practical-MIP}} \\
		\bottomrule
	\end{tabular}
	\caption{Number of paraboloids to approximate the zigzag function from below.}
	\label{tab:number-paraboloids-zigzag-below}
\end{table}

\begin{figure}
\caption{Approximation of the zigzag function from below by nine paraboloids with $\eps = 10^{-1}$.}
\label{fig:zigzag-approx-below}
\centering 
\begin{tikzpicture}
	\begin{axis}[
		axis lines=middle,
		xlabel={$x$},
		ylabel={$f(x)$},
		grid=major,
		width=12cm,
		height=8cm,
		xtick={-5,-4,-3,-2,-1,0,1,2,3,4,5},
		ytick={-1.5,-1,-0.5,0,0.5,1},
		clip mode=individual,
		ymin=-1.8, ymax=0.6
		]
		\addplot[color=black, line width=2pt,mark size=0pt] coordinates {
			(-5,-0.12801019571599248)
			(-4.964333030490388,-0.12446757554744907)
			(-4.523363861798294,-0.1946982637766576)
			(-4.186332389004458,-0.3937836755004518)
			(-3.563254132317301,-0.6434453088250474)
			(-3.289095129965463,-0.5770254473388788)
			(-2.7552444566311944,-0.9671849317937837)
			(-2.2368021165781053,-1.2943844007567447)
			(-1.4493203202395986,-0.7368862021423581)
			(-0.9500258512314901,-0.3908137365219214)
			(-0.8611768289925197,-0.38988151656613956)
			(-0.7865431896495103,-0.4006105011239624)
			(-0.6809775831455038,-0.4793286685566136)
			(-0.080199727256694,-0.808540514189117)
			(0.035676500210189144,-0.8733602373302513)
			(0.3920045223634554,-0.8963166811449923)
			(0.601730316365771,-0.8374228572870663)
			(1.0899694535620053,-0.832309313959877)
			(1.482993178169379,-0.6014963419021636)
			(2.0671973152584076,-0.9960694536135861)
			(2.770942138399492,-0.34221861821876076)
			(3.2759504159580066,0.05120309491708941)
			(3.6241479321024053,0.09796193184439114)
			(4.0574184357655545,0.043150841431488146)
			(4.8362120289126835,0.09860744585003525)
			(5,0.11308897556235686)
		};
		\foreach \qcoeff/\lcoeff/\constcoeff in {
			-0.55727/-5.44384/-13.443479441201276,
			-1.46451/-2.57902/-1.5600607083755302,
			-0.92229/-6.34842/-11.55436687072635,
			-0.04825/0.06235/-0.9187732533645758,
			-0.39455/-0.62117/-0.8558836424288132,
			-0.49335/3.66264/-6.7191414002703045,
			-0.12555/1.24672/-3.0367732803589864,
			-0.13068/-1.2964/-3.603633754061471,
			-1.31648/3.83179/-3.47147977351185
		} {
			\addplot[domain=-5:5, smooth, variable=\x, draw=none, fill=gray!30, fill opacity=1.0, forget plot]
			({\x},{\qcoeff*\x*\x + \lcoeff*\x + \constcoeff}) -- (axis cs:5,-1.8) -- (axis cs:-5,-1.8) -- cycle;
			
			\addplot[domain=-5:5, dashed, variable=\x, gray]
			({\x}, {\qcoeff*\x*\x + \lcoeff*\x + \constcoeff});
		}
	\end{axis}
\end{tikzpicture}

\end{figure}

\Cref{tab:number-paraboloids-exp-above} provides a summary for the results for $\exp$ approximated from above.  
Comparing the successful runs for either method, \texttt{practical-MIP} returns an approximation more frequently than \texttt{exact-MIP} and computes approximations relying on less paraboloids,
which has already been observed for the zigzag function.
However, either method struggles with the combination of a high accuracy and a large Lipschitz constant, both of which yield large \ac{mip} problems to solve. 
As noted above, the Lipschitz constant is computed as $\exp(\overline{x})$, where $\overline{x}$ is the right bound of the domain.
If $\overline{x} = 5$, the increase in the problem size causes the \ac{mip} problems to be intractable.
This is also true for the approximations from below, see~\Cref{tab:number-paraboloids-exp-below}, except for cases with low accuracy on the smallest domain, which we account to the convexity of the epigraph of $\exp$ there.
As the described effects are similar for approximating $x^3$, we avoid further detailed discussion at this point and refer to~\Cref{tab:number-paraboloids-x3-above} and~\Cref{tab:number-paraboloids-x3-below} in~\Cref{subsec:approx-x3}.

\begin{table}[t]
	\centering
	\begin{tabular}{ccccccccccc}
		\toprule
		& & \multicolumn{9}{c}{$\eps$} \\ \cline{3-11}
		& & \multicolumn{9}{c}{\vspace{-0.28cm}} \\
		$\domain$ & & $10^{0}$ & $10^{-1}$ & $10^{-2}$ & $10^{-3}$  & & $10^{0}$ & $10^{-1}$ & $10^{-2}$ & $10^{-3}$ \\ \cline{3-6}\cline{8-11}
		\\[-1em] 
		$[-5, -2]$ &  & 1 & 1 & 3 & - &  & 1 & 1 & 2 & 5 \\ 
		\\[-1em] 
		$[-2, \phantom{-}2]$ &  & 3 & - & - & - &  & 2 & 5 & 32 & - \\ 
		\\[-1em] 
		$[-5, \phantom{-}2]$ &  & 4 & - & - & - &  & 3 & 8 & 81 & - \\ 
		\\[-1em] 
		$[\phantom{-}2, \phantom{-}5]$ &  & - & - & - & - &  & - & - & - & - \\ 
		\\[-1em] 
		$[-2, \phantom{-}5]$ &  & - & - & - & - &  & - & - & - & - \\ 
		\\[-1em] 
		$[-5, \phantom{-}5]$ &  & - & - & - & - &  & - & - & - & - \\ 
		& & \multicolumn{4}{c}{\texttt{exact-MIP}} & & \multicolumn{4}{c}{\texttt{practical-MIP}} \\
		\bottomrule
	\end{tabular}
	\caption{Number of paraboloids to approximate $\exp$ from above.}
	\label{tab:number-paraboloids-exp-above}
\end{table}
\begin{table}[t]
	\centering
	\begin{tabular}{ccccccccccc}
		\toprule
		& & \multicolumn{9}{c}{$\eps$} \\ \cline{3-11}
		& & \multicolumn{9}{c}{\vspace{-0.28cm}} \\
		$\domain$ & & $10^{0}$ & $10^{-1}$ & $10^{-2}$ & $10^{-3}$  & & $10^{0}$ & $10^{-1}$ & $10^{-2}$ & $10^{-3}$ \\ \cline{3-6}\cline{8-11}
		\\[-1em] 
		$[-5, -2]$ &  & 1 & 1 & 3 & - &  & 1 & 1 & 2 & 8 \\ 
		\\[-1em] 
		$[-2, \phantom{-}2]$ &  & 3 & - & - & - &  & 2 & 4 & 20 & - \\ 
		\\[-1em] 
		$[-5, \phantom{-}2]$ &  & 3 & - & - & - &  & 2 & 6 & 56 & - \\ 
		\\[-1em] 
		$[\phantom{-}2, \phantom{-}5]$ &  & - & - & - & - &  & 5 & 20 & - & - \\ 
		\\[-1em] 
		$[-2, \phantom{-}5]$ &  & - & - & - & - &  & - & - & - & - \\ 
		\\[-1em] 
		$[-5, \phantom{-}5]$ &  & - & - & - & - &  & - & - & - & - \\ 
		& & \multicolumn{4}{c}{\texttt{exact-MIP}} & & \multicolumn{4}{c}{\texttt{practical-MIP}} \\
		\bottomrule
	\end{tabular}
	\caption{Number of paraboloids to approximate $\exp$ from below.}
	\label{tab:number-paraboloids-exp-below}
\end{table}

At first sight, these results seem to show tremendous limitations for parabolic approximations. 
But, one must recognize that the constructed \ac{mip} models basically rely on a Lipschitz property only and not on other properties like differentiability.
This provides freedom to approximate general functions such as the zigzagging one, for instance, but is paid for with increased computational effort and thus limited success on specific functions.

Turning to the results for $\sin$ in~\Cref{tab:number-paraboloids-sin-above} and~\Cref{tab:number-paraboloids-sin-below}, we remark a larger number of successful runs over all. 
The problem size seems to be controllable up to $\eps = 10^{-1}$ for \texttt{exact-MIP} and $\eps = 10^{-2}$ for \texttt{practical-MIP}.
The latter approach is even successful for higher accuracy in some cases. 
In comparison to the zigzag function, for which the same Lipschitz constant of 1 is used, the methods are able to compute approximations more frequently.
Beside the slightly smaller domain size, we assume that the regularity of $\sin$ facilitates finding adequate paraboloids, see~\Cref{fig:approx-sin-below-e-2} for a visualization of this intuition.

\begin{figure}
	\caption{Approximation  with $\eps = 10^{-2}$ of $\sin$ on $[-\pi/2, 3\pi/2]$ from below by 13 paraboloids.}
	\label{fig:approx-sin-below-e-2}
	\centering
	\begin{tikzpicture}[scale=1.8]
		\clip (-1.57,-1.2) rectangle (4.71,1.4);
		\ 
		\draw[->] (-1.57,0) -- (4.71,0) node[right] {$x$};
		\draw[->] (0,-1.2) -- (0,1.4) node[above] {$y$};
		\foreach \x in {-1.5,-1.0,...,4.5}
		\draw (\x,-0.05) -- (\x,0.05);
		\foreach \y in {-1.0,-0.5,...,1.0}
		\draw (-0.05,\y) -- (0.05,\y);
		\foreach \x/\label in {-1.5,-1.0,-0.5,0.5,1.0,1.5,2.0,2.5,3.0,3.5,4.0,4.5}
		\draw (\x,0) node[below] {$\label$};
		\foreach \y in {-1.0,-0.5,0.5,1.0}
		\draw (0,\y) node[left] {$\y$};
		\draw[domain=-1.57:4.71,smooth,variable=\x,black, line width=2pt] plot ({\x},{sin(deg(\x))});
		\def\quadraticcoeffs{{-0.0186,-0.21609,-0.08009,-0.16669,-0.05161,-0.46946,-0.35723,-0.45377,-0.29585,-0.27251,-0.11547,-0.14214,-0.40536,}}
		\def\linearcoeffs{{0.05843,0.67886,0.25162,0.52368,0.16136,1.47624,1.12229,1.42556,0.92943,0.85611,0.36276,0.44654,1.27349,}}
		\def\constantcoeffs{{-0.868915,-0.08614,-0.517885,-0.196032,-0.66694,-0.167049,-0.011553,-0.132701,-0.003962,-0.016697,-0.364584,-0.269412,-0.056143,}}
		\foreach \i in {0,...,12}{
			\draw[domain=-1.57:4.71,smooth,variable=\x,gray] plot ({\x},{\quadraticcoeffs[\i]*\x*\x + \linearcoeffs[\i]*\x + \constantcoeffs[\i]});
		}
	\end{tikzpicture}
\end{figure}

\begin{table}[h]
	\centering
	\begin{tabular}{ccccccccccc}
		\toprule
		& & \multicolumn{9}{c}{$\eps$} \\ \cline{3-11}
		& & \multicolumn{9}{c}{\vspace{-0.28cm}} \\
		$\domain$ & & $10^{0}$ & $10^{-1}$ & $10^{-2}$ & $10^{-3}$  & & $10^{0}$ & $10^{-1}$ & $10^{-2}$ & $10^{-3}$ \\ \cline{3-6}\cline{8-11}
		\\[-1em] 
		$[-\pi/2, \phantom{3}\pi/2]$ &  & 2 & 4 & - & - &  & 1 & 3 & 9 & - \\ 
		\\[-1em] 
		$[\phantom{-}\pi/2, 3\pi/2]$ &  & 2 & 4 & - & - &  & 1 & 3 & 8 & - \\ 
		\\[-1em] 
		$[-\pi/2, 3\pi/2]$ &  & 3 & 9 & - & - &  & 1 & 5 & 47 & - \\ 
		\\[-1em] 
		$[0, \phantom{2}\pi]$ &  & 1 & 3 & - & - &  & 1 & 1 & 5 & 40 \\ 
		\\[-1em] 
		$[\pi, 2\pi]$ &  & 1 & 2 & - & - &  & 1 & 1 & 3 & 32 \\ 
		\\[-1em] 
		$[0, 2\pi]$ &  & 3 & 7 & - & - &  & 2 & 5 & 24 & - \\ 
		& & \multicolumn{4}{c}{\texttt{exact-MIP}} & & \multicolumn{4}{c}{\texttt{practical-MIP}} \\
		\bottomrule
	\end{tabular}
	\caption{Number of paraboloids to approximate $\sin$ from above.}
	\label{tab:number-paraboloids-sin-above}
\end{table}
\begin{table}[h]
	\centering
	\begin{tabular}{ccccccccccc}
		\toprule
		& & \multicolumn{9}{c}{$\eps$} \\ \cline{3-11}
		& & \multicolumn{9}{c}{\vspace{-0.28cm}} \\
		$\domain$ & & $10^{0}$ & $10^{-1}$ & $10^{-2}$ & $10^{-3}$  & & $10^{0}$ & $10^{-1}$ & $10^{-2}$ & $10^{-3}$ \\ \cline{3-6}\cline{8-11}
		\\[-1em] 
		$[-\pi/2, \phantom{3}\pi/2]$ &  & 2 & 4 & - & - &  & 1 & 3 & 8 & - \\ 
		\\[-1em] 
		$[\phantom{-}\pi/2, 3\pi/2]$ &  & 2 & 4 & - & - &  & 1 & 3 & 8 & - \\ 
		\\[-1em] 
		$[-\pi/2, 3\pi/2]$ &  & 2 & 5 & - & - &  & 1 & 3 & 13 & - \\ 
		\\[-1em] 
		$[0, \phantom{2}\pi]$ &  & 1 & 2 & - & - &  & 1 & 1 & 3 & 16 \\ 
		\\[-1em] 
		$[\pi, 2\pi]$ &  & 1 & 3 & - & - &  & 1 & 1 & 5 & - \\ 
		\\[-1em] 
		$[0, 2\pi]$ &  & 3 & 7 & - & - &  & 2 & 5 & 44 & - \\ 
		& & \multicolumn{4}{c}{\texttt{exact-MIP}} & & \multicolumn{4}{c}{\texttt{practical-MIP}} \\
		\bottomrule
	\end{tabular}
	\caption{Number of paraboloids to approximate $\sin$ from below.}
	\label{tab:number-paraboloids-sin-below}
\end{table}

In comparison to approximating $\exp$ or $x^3$, we highlight the moderate number of paraboloids necessary for $\sin$.
This implies a small number of additional constraints when applying the parabolic relaxation to trigonometric functions.
We analyze the resulting effects on a distinct set of instances in the following section. 

\subsection{Applied Parabolic Relaxation}
\label{subsec:comp-para-relax}

To make use of the results above, we perform the parabolic relaxation on feasible instances from the MINLPLib~\cite{MINLPLib} available in the OSIL-format.
Due to this format, we can conveniently rewrite the instances as an expression tree in terms of quadratic parts and univariate functions such that we can identify parts to substitute, see \Cref{subsec:comp-para-approx}. 
Since there exist results for approximation with $\eps = 10^{-2}$ throughout the majority of relevant functions and domains, we fix this accuracy in the current section. 
For every instance, we scan the expression tree formulation and identify a substitutable nonlinear function if its domain is a subset of the ones computed before and the domain width is nontrivial, \ie, $\abs{\domain} > 0$.
We summarize all such instances as \texttt{all}.
As discussed above, the substitution of $\sin$ or $\cos$ requires only a few paraboloids in comparison to $\exp$ or $x^3$. 
Therefore, we define the set \texttt{sub} as a filter on \texttt{all}, where we consider either instances with $\sin$ and $\cos$, where there exist approximations for all nonlinear functions, or with $\exp$ and $x^3$, where there does not exist approximations for all nonlinear functions.

For each instance in the sets \texttt{all} and \texttt{sub}, we consider three types of problems:
First, we take the instance in its original form into account, as given in \software{GAMS}-format on the MINLPLib website (\texttt{orig}).
Second, we relax every possible one dimensional nonlinear function (\texttt{para}).
For this, we follow the procedure described in~\Cref{subsec:para-relax}.
In particular, we reformulate each instance in terms of one-dimensional functions and linear/quadratic parts.
This is complemented by an expression-tree-based bound propagation of the natural bounds of the respective arguments.
Third, we keep the original functions and additionally incorporate the parabolic relaxation, aiming for the solvers to refine their relaxations with the quadratic cuts during the process (\texttt{both}).
For details, we again refer to~\Cref{subsec:para-relax}.

We tackle these problems with \software{Gurobi} and \software{SCIP}.
Running on nodes with \SI{64}{\giga\byte} of RAM, we limit the memory used by the solvers to \SI{60}{\giga\byte}.
Further, we apply a time limit of \SI{4}{\hour} (\SI{14400}{\second}) and bound the usable threads by eight.
Note that we experience numerical issues on one instance for each solver, which we exclude in the respective evaluation.
An overview of all runs can be found in~\Cref{tab:overview} in~\Cref{subsec:overview-minlplib}, where we highlighted the excluded instance with an asterisk at the respective solver entry.

For comparison, we consider the optimality gap as a suitable metric. 
We leverage a definition that it widely used, \eg, by the solver CPLEX, and give it in absolute and relative terms.
Formally, let $c^*$ denote the best known value for~\labelcref{prob:minlp} and $d$ as a given lower/dual bound.
The \emph{absolute optimality gap} is then defined as
\begin{equation*}
	\mathrm{absgap} = \abs{c^* - d},
\end{equation*}
and its relative analog  as
\begin{equation*}
	\mathrm{relgap} = \abs{c^* - d}/(\abs{c^*} + 10^{-10})= \mathrm{absgap}/(\abs{c^*} + 10^{-10}).
\end{equation*}
This definition is independent of the presentation of~\labelcref{prob:minlp} as a maximization or minimization problem.
The addition of $10^{-10}$ avoids a division by zero.
To keep the evaluation results for each solver independent of other solvers, we note that for each instance, $c^*$ is defined as the best value among the one given on the MINLPLib website and the one returned by the current solver.

Since the results from~\Cref{subsec:comp-para-approx} serve as an approximation with accuracy of $\eps = 10^{-2}$, one cannot expect to have an optimality gap of zero on problems of type \texttt{para}. 
However, the method is expected to achieve comparable gaps with the returned dual value on a smaller time scale. 
Running \software{Gurobi} on \texttt{all} with type \texttt{orig} and \texttt{para}, see~\Cref{tab:statistics-gurobi-all}, we can unfortunately not observe such an effect. 
However, when executing \software{SCIP} on the same problems, we can see a significant decrease in the running times in the median, while maintaining comparable gaps, see~\Cref{tab:statistics-scip-all}.
This effect even increases when restricting on \texttt{sub}, which has been declared before to be a more promising set of instances for the parabolic relaxation, see~\Cref{tab:statistics-scip-sub}.

\begin{table}[h]
	\begin{tabular}{lcccccc}
		\toprule
		& & \multicolumn{2}{c}{\texttt{orig}} & & \multicolumn{2}{c}{\texttt{para}} \\ \cline{3-4}\cline{6-7}
		\\[-1em] 
		& & time in s & relgap & & time in s & relgap \\
		No. instances & & 38 & 38 && 38 & 38 \\ \midrule
		Minimum &  & 0.01 & 0.0000 &  & 0.01 & 0.0000\\
		1st Quartile &  & 0.08 & 0.0000 &  & 0.51 & 0.0024\\
		Median &  & 15.97 & 0.0001 &  & 102.07 & 0.0652\\
		3rd Quartile &  & limit & 0.5016 &  & limit & 0.5381\\
		Maximum &  & limit & inf &  & limit & inf\\
		\bottomrule
	\end{tabular}
	\caption{Statistics for \software{Gurobi} on \texttt{all}.}
	\label{tab:statistics-gurobi-all}
\end{table}

\begin{table}[h]
	\begin{tabular}{lcccccc}
		\toprule
		& & \multicolumn{2}{c}{\texttt{orig}} & & \multicolumn{2}{c}{\texttt{para}} \\ \cline{3-4}\cline{6-7}
		\\[-1em] 
		& & time in s & relgap & & time in s & relgap \\
		No. instances & & 38 & 38 && 38 & 38 \\ \midrule
		Minimum &  & 0.01 & 0.0000 &  & 0.01 & 0.0001\\
		1st Quartile &  & 0.49 & 0.0000 &  & 0.57 & 0.0049\\
		Median &  & 1774.99 & 0.0001 &  & 51.10 & 0.0389\\
		3rd Quartile &  & limit & 0.1405 &  & limit & 0.8334\\
		Maximum &  & limit & 6.5758 &  & limit & inf\\
		\bottomrule
	\end{tabular}
	\caption{Statistics for \software{SCIP} on \texttt{all}.}
	\label{tab:statistics-scip-all}
\end{table}

\begin{table}[h]
	\begin{tabular}{lcccccc}
		\toprule
		& & \multicolumn{2}{c}{\texttt{orig}} & & \multicolumn{2}{c}{\texttt{para}} \\ \cline{3-4}\cline{6-7}
		\\[-1em] 
		& & time in s & relgap & & time in s & relgap \\
		No. instances & & 23 & 23 && 23 & 23 \\ \midrule
		Minimum &  & 0.03 & 0.0000 &  & 0.04 & 0.0001\\
		1st Quartile &  & 0.39 & 0.0000 &  & 0.57 & 0.0116\\
		Median &  & 915.51 & 0.0001 &  & 25.70 & 0.0315\\
		3rd Quartile &  & limit & 0.0895 &  & limit & 0.5173\\
		Maximum &  & limit & 1.0000 &  & limit & inf\\
		\bottomrule
	\end{tabular}
	\caption{Statistics for \software{SCIP} on \texttt{sub}.}
	\label{tab:statistics-scip-sub}
\end{table}

This motivates to dive deeper in the latter case. 
After removing all instances, on which \software{SCIP} requires less than \SI{5}{\second} on \texttt{orig} and \texttt{para}, \Cref{tab:overview-filtered} presents an overview of the remaining times and gaps.
In the lower half, instances are gathered, on which the solver hits the time limit and returns comparable gaps for both problem types with one exception (lnts400).
Turning to the instances in the upper half, in all cases except the first, \software{SCIP} on type \texttt{para} requires a fraction of the solution time in comparison to \texttt{orig}.
Except for instance t1000 the relative optimality gaps produced on the former is below $3.9$\%, which we consider quite small. 
Note that for t1000, the relative gap is  high in comparison, but in absolute terms it is below 0.0001. 
This effect is caused by the optimal objective of zero and the corresponding division by a small term, compare the definition of relgap.
Overall, these results are already promising and demonstrate the ability of \software{SCIP} to handle parabolic constraints efficiently.

\begin{table}
	\begin{tabular}{lccccc}\toprule
		instance & \multicolumn{2}{c}{\texttt{orig}} &  & \multicolumn{2}{c}{\texttt{para}}\\ \cline{2-3} \cline{5-6} 
		\\[-1em]
		& run time in s & relgap &  & run time in s & relgap\\ \hline\hline 
		ghg\_1veh & 19.44 & 0.0001 &  & 25.70 & 0.0003\\ 
		pooling\_epa1 & 53.93 & 0.0001 &  & 9.42 & 0.0389\\ 
		kriging[...]red020 & 773.66 & 0.0001 &  & 19.23 & 0.0224\\ 
		t1000 & 915.51 & 0.0000 &  & 0.57 & 0.9990\\ 
		lnts50 & limit & 0.0258 &  & 51.10 & 0.0228\\ 
		pooling\_epa2 & limit & 0.0151 &  & 2047.92 & 0.0151\\ 
		lnts100 & limit & 0.0345 &  & 3889.47 & 0.0315\\ 
		\midrule 
		contvar & limit & 0.5024 &  & limit & 0.5031\\ 
		ex8\_4\_6 & limit & 1.0000 &  & limit & 1.0000\\ 
		feedtray & limit & 0.8048 &  & limit & 0.5762\\ 
		ghg\_2veh & limit & 0.0447 &  & limit & 0.0426\\ 
		ghg\_3veh & limit & 0.5200 &  & limit & 0.5173\\ 
		lnts200 & limit & 0.0895 &  & limit & 0.0864\\ 
		lnts400 & limit & 0.1153 &  & limit & inf\\ 
		pooling\_epa3 & limit & 0.0041 &  & limit & 0.0309\\ 
		\bottomrule
	\end{tabular}
	\caption{{\texttt{SCIP} on \texttt{sub} with run time $>$ \SI{5}{\second}.}}
	\label{tab:overview-filtered}
\end{table}

This realization encourages to compare the performance of \software{SCIP} on \texttt{orig} and \texttt{both} types.
Now, gaps of zero can be expected and, thus, a comparison in terms of sole running time is possible.
This allows to use the metrics shifted geometric mean (SGM) of run time with a shift of \SI{10}{\second}~\cite{Achterberg2007} and performance profiles~\cite{Siqueira2016perprof}.
On instance set \texttt{all}, this results in a SGM of \SI{530.6}{\second} for \texttt{orig} and \SI{429.2}{\second} for \texttt{both}, where a virtual best of these types yields \SI{379.3}{\second}.
If we restrict the instance set to \texttt{sub}, the corresponding values are \SI{535.6}{\second}, \SI{319.3}{\second}, and \SI{303.3}{\second}, respectively.
This demonstrates that the instances equipped with parabolic approximations on top of the original constraints are solved faster on average.
In particular, the solving times are reduced by more than 40\% when applied to an instance class, which can be identified a-priori.

The acceleration of \software{SCIP} by our approach is confirmed when looking at the performances profiles in~\Cref{fig:perfprof-all} for \texttt{all}.
A certain price in terms of time is paid due to the overhead of additional parabolic constraints on instances, which are originally solved fast.
However, \software{SCIP} seems to leverage the paraboloids for the other instances in order to boost the solving process significantly.
This effect is even enhanced when tackling the promising instances \texttt{sub}, see~\Cref{fig:perfprof-sub}.

In general, the computations indicate a huge potential to accelerate the solving process of \ac{minlp} problems by means of paraboloids. 
In~\Cref{subsec:para-relax}, we have laid out how additional parabolic constraints do not increase the complexity of the underlying problem.
This suggests to interpret parabolic constraints as quadratic cuts. 
Now, when solving \ac{mip} problems in practice, a solver does not incorporate all potential cuts at the root node, but applies them to strengthen the continuous relaxation adaptively throughout the branch-and-bound scheme.
Consequently, in \ac{minlp}, we strongly believe that there is still room for improvement of this approach by integrating the parabolic cuts into a spatial branch-and-bound framework,
since it keeps intermediate problem sizes low while providing tight relaxations selectively.

\begin{figure}[t]
	\begin{tikzpicture}
		\begin{semilogxaxis}[const plot,
			cycle list={
				solid,
				{dashed},
				{dotted}},
			xmin=1, xmax=4179.50,
			ymin=-0.003, ymax=1.003,
			ymajorgrids,
			ytick={0,0.2,0.4,0.6,0.8,1.0},
			xlabel={Performance ratio},
			ylabel={Percentage of problems solved},title={Performance Profile},
			legend pos={south east},
			width=\textwidth
			]
			\addplot+[mark=none, thick] coordinates {
				(1.0000,0.1538)
				(1.3297,0.1538)
				(1.4737,0.1795)
				(1.5833,0.2051)
				(1.6842,0.2308)
				(1.8333,0.2564)
				(2.0000,0.2821)
				(2.0870,0.2821)
				(2.1026,0.3077)
				(2.3333,0.3333)
				(2.5538,0.3590)
				(2.6347,0.3846)
				(3.1224,0.4103)
				(3.2500,0.4359)
				(3.6550,0.4615)
				(6.0000,0.4872)
				(6.1231,0.5128)
				(19.8000,0.5385)
				(4179.5022,0.5385)
			};
			\addlegendentry{scip-both}
			\addplot+[mark=none, thick] coordinates {
				(1.0000,0.4359)
				(1.3297,0.4615)
				(2.0000,0.4615)
				(2.0870,0.4872)
				(19.8000,0.4872)
				(24.7097,0.5128)
				(3980.4783,0.5385)
				(4179.5022,0.5385)
			};
			\addlegendentry{scip-orig}
		\end{semilogxaxis}
	\end{tikzpicture}
	\caption{\software{SCIP} on \texttt{all} with \texttt{orig} and \texttt{both}.}
	\label{fig:perfprof-all}
\end{figure}

\begin{figure}
	\begin{tikzpicture}
		\begin{semilogxaxis}[const plot,
			cycle list={
				solid,
				{dashed},
				{dotted}},
			xmin=1, xmax=4179.50,
			ymin=-0.003, ymax=1.003,
			ymajorgrids,
			ytick={0,0.2,0.4,0.6,0.8,1.0},
			xlabel={Performance ratio},
			ylabel={Percentage of problems solved},title={Performance Profile},
			legend pos={south east},
			width=\textwidth
			]
			\addplot+[mark=none, thick] coordinates {
				(1.0000,0.2174)
				(1.3297,0.2174)
				(1.5833,0.2609)
				(1.6842,0.3043)
				(2.0000,0.3478)
				(2.1026,0.3913)
				(2.3333,0.4348)
				(2.6347,0.4783)
				(3.1224,0.5217)
				(3.2500,0.5652)
				(19.8000,0.6087)
				(4179.5022,0.6087)
			};
			\addlegendentry{scip-both}
			\addplot+[mark=none, thick] coordinates {
				(1.0000,0.3913)
				(1.3297,0.4348)
				(19.8000,0.4348)
				(24.7097,0.4783)
				(3980.4783,0.5217)
				(4179.5022,0.5217)
			};
			\addlegendentry{scip-orig}
		\end{semilogxaxis}
	\end{tikzpicture}
	\caption{\software{SCIP} on \texttt{sub} with \texttt{orig} and \texttt{both}.}
	\label{fig:perfprof-sub}
\end{figure}

\section{Conclusion \& Future Work}
\label{sec:conclusion}

We proposed a quadratic approximation approach for solving general \acf{minlp} problems. 
Our method globally approximates nonlinear constraint functions with paraboloids and directly incorporates the resulting parabolic constraints into the original \ac{minlp} formulation. 
To construct these approximations, we combined tailored search methods with a \acf{mip} model, for which we provide theoretical verification. 
Moreover, we presented two main strategies for effectively integrating these approximations into \ac{minlp} models. 
We tested both strategies - approximation and incorporation - computationally on common functions and several instances from the MINLPLib, respectively, demonstrating promising numerical results.

We see significant potential in combining our parabolic approximation approach with established optimization techniques. 
In fact, adaptively incorporating parabolic constraints can be interpreted as a nonlinear counterpart to linear cutting-plane methods such as Gomory cuts. 
Hence, it appears particularly promising to embed our approach within spatial branch-and-bound frameworks, enabling concatenation of approximation results, domain splitting, and adaptive refinement of quadratic approximations. 
Moreover, combining parabolic approximations with adaptive cutting methods could further tighten relaxations, reduce initial problem sizes, and yield improved dual bounds.

Although our experiments showed clear improvements and practical benefits, the approximation scheme still offers room for enhancement.
Considering the analysis of expression trees, we suggest specializing the approximation scheme for one-dimensional functions by leveraging additional structural properties, such as differentiability. 
Such targeted approximations would allow tackling larger domains and enable smaller accuracies, potentially further improving computational performance.
In conclusion, our method provides a flexible and complementary tool for \ac{minlp} problems and suggests several promising directions for future research.

\begin{acronym}[MIP]
	\acro{mip}[{MIP}\xspace]{Mixed-Integer Linear Programming}	
	\acro{mipm}[{MIP}\xspace]{Mixed-Integer Linear Program}
\end{acronym}
\begin{acronym}[NLP]
	\acro{nlp}[{NLP}\xspace]{Nonlinear Programming}
\end{acronym}
\begin{acronym}[MIQCP]
	\acro{miqcp}[{MIQCP}\xspace]{Mixed-Integer Quadratically-Constrained Programming}
	\acro{miqcpm}[{MIQCP}\xspace]{Mixed-Integer Quadratically-Constrained Program}
\end{acronym}
\begin{acronym}[MINLP]
	\acro{minlp}[{MINLP}\xspace]{Mixed-Integer Nonlinear Programming}
	\acro{minlpm}[{MINLP}\xspace]{Mixed-Integer Nonlinear Program}
\end{acronym}
\begin{acronym}[CARP]
	\acro{carp}[CARP]{Candid Approximation and Relaxation by Paraboloids}
\end{acronym}
\newpage
\printbibliography
\newpage

\appendix
\section{Proofs for Lemmata in \Cref{sec:approx}}

\subsection{Proof of~\Cref{thm:upper-bound-K}}
\label{subsec:proof-thm-existence}

\begin{proof}
	We want to prove the statement constructively. 
	In particular, for every $\vt \in \mathcal{G}_\eps$, we define an explicit paraboloid $p$ such that all of these paraboloids fulfill the constraints of problem~\labelcref{problem:original-multi-dim}.
	As the grid widths are chosen identically, we can interchange $\vt \in \mathcal{G}_\eps$ and $\vd \in \mathcal{G}$.
	
	First, see that 
	\begin{equation*}
		\sum_{i = 1}^n \Delta t_i = n \Delta \leq \frac{n+1}{n} \frac{\eps - \delta}{3L}.
	\end{equation*}
	Hence, the choice of $\Delta t_i$ is in accordance with the required property~\labelcref{eq:choice-delta-ti}.
	The property~\labelcref{eq:choice-delta-di} follows by construction.
	
	Now, consider an arbitrary but fixed vector $\vt \in \mathcal{G}_\eps$.
	Then, the vertices $\vv$ of ${[\vt - \vect{\Delta}, \vt + \vect{\Delta}]}$ are exactly the elements in $\mathcal{N}(\vt)$. 
	This allows for a parameterized representation of the vertices as $\vv_\vect{u} \coloneqq \vt + \Delta \sum_{i = 1}^n u_i \vect{e}_i$ for all $\vect{u} \in \{-1, 1\}^n$, where $\vect{e}_i$ denotes the $i$th unit vector.
	
	Notice that a paraboloid $p(\vx) = \sum_{i = 1}^n \alpha_i x_i^2 + \beta_i x_i + \gamma$ has $2n + 1$ coefficients.
	Hence, it is uniquely determined by solving the following system of equations:
	\begin{align}
		\label{eq:define-paraboloid-1}
		p(\vt) &= f(\vt) - \delta, \\
		\label{eq:define-paraboloid-2}
		\frac{\mathrm{d}}{\mathrm{d}x_i} p(\vv_\vect{u}) &= -u_i 2L, &\text{ for all } \vect{u} \in \{-1, 1\}^n ,i \in [n].
	\end{align}
	The $2^n$ many equations~\labelcref{eq:define-paraboloid-2} can be rewritten as the system
	\begin{equation}
		\label{eq:define-paraboloid-2-solved}
		2\alpha_i (t_i + \Delta)  + \beta_i = -2 L \qquad \land \qquad 2\alpha_i (t_i - \Delta) + \beta_i = 2L,
	\end{equation}
	for all $i \in [n]$, which has $2n$ equations.
	Solving~\labelcref{eq:define-paraboloid-2-solved}, we receive $\alpha_i = -L/\Delta$ and $\beta_i =2L t_i /\Delta $. 
	As $\gamma$ is the only degree of freedom left, we compute it by solving~\labelcref{eq:define-paraboloid-1}. 
	Therefore, such a $p$ exists and is uniquely determined. 
	As $p$ fulfills equation~\labelcref{eq:define-paraboloid-1} for $\vt$, we can set $s_\vt = 1$ for this $p$ and, thus, with \labelcref{eq:define-paraboloid-2}, our construction satisfies \labelcref{eq:below-approx-multi-dim-1}-\labelcref{eq:below-approx-multi-dim-3}.
	
	In order to prove conformity with constraints~\labelcref{eq:above-approx-multi-dim} and~\labelcref{eq:integral-tracking-multi-dim}, we abstract the situation and assume $\vt = \vd = \vect{0}$. 
	Following the computations from before, it is ${\gamma = f(\vect{0}) - \delta}$, $\alpha_i = -L/\Delta$ and $\beta_i = 0$. 
	Now, if we show 
	\begin{equation}
		\label{eq:transfer-t-0}
		f(\vx) - \nu\eps - p(\vx) \geq 0, \text{ for all } \vx \in [-\Delta, \Delta], 
	\end{equation}
	this transfers to the original box $[\vd - \vect{\Delta}, \vd + \vect{\Delta}]$ and directly gives~\labelcref{eq:above-approx-multi-dim}.
	Furthermore, inequality~\labelcref{eq:transfer-t-0} implies that the argument of the integral in~\labelcref{eq:integral-tracking-multi-dim} is non-positive, which implies $v_\vd = 0$ if combined with the objective function.
	
	To show~\labelcref{eq:transfer-t-0}, we remark that the Lipschitz continuity of $f$ at a point $\vx$ gives $f(\vx) \geq f(\vect{0}) - L \norm[1]{\vx} \eqqcolon \Lambda(\vx)$.
	Then, the (sub)gradient $\nabla \Lambda(\vx) = -L\, \mathrm{sgn}(\vx)$, where $\mathrm{sgn}(\vx)$ is the component-wise sign function. 
	Now, defining the auxiliary function $g(\vx) \coloneqq \Lambda(\vx) - \nu\eps - p(\vx)$, we want to show $\min_\vx g(\vx) \geq 0$.
	From~\labelcref{eq:define-paraboloid-1}, we derive
	\begin{equation*}
		g(\vect{0}) = \Lambda(\vect{0}) - \nu\eps - p(\vect{0}) = f(\vect{0}) - \delta/2 - f(\vect{0}) + \delta > 0.
	\end{equation*}
	As $\alpha_i < 0$, it also holds true that
	\begin{equation*}
		\lim\limits_{\norm[1]{\vx} \rightarrow \infty} g(\vx) = \lim\limits_{\norm[1]{\vx} \rightarrow \infty} f(\vect{0}) -L\norm[1]{\vx} - \nu\eps - p(\vx) = \infty.
	\end{equation*}
	Hence, suitable candidates for achieving $\min_\vx g(\vx)$ must fulfill $\nabla g(\vx) = \vect{0}$, which gives
	\begin{equation*}
		\vect{0} = \nabla g(\vx) = -L\, \mathrm{sgn}(\vx) + (2L/\Delta) \vx.
	\end{equation*}
	For a solution $\vx'$ to this equation, it holds true that $x'_i = \pm \Delta/2$ for $i \in [n]$. 
	Evaluating $p$ with the coefficients computed before at $\vx'$ gives
	\begin{equation*} 
		p(\vx') = \sum_{i = 1}^n (-L/\Delta) (\Delta/2)^2 + f(\vect{0}) - \delta = -nL \Delta /4 + f(\vect{0}) - \delta.
	\end{equation*}
	The respective value of $g$ is 
	\begin{align*} 
		g(\vx') &= f(\vect{0}) - L\norm[1]{\vx'} - \nu\eps - p(\vx') \\
		& = f(\vect{0}) - Ln\Delta/2 - \nu\eps + nL \Delta / 4 - f(\vect{0}) + \delta \\
		& = \delta/2- nL\Delta /4 \geq \delta/2 - nL \delta/(2nL) = 0,
	\end{align*}
	using the choice $\nu = \delta/(2\eps)$ and $\Delta \leq (2\delta)/(nL)$.
	In summary, this shows that ${\min_{\vx}g(\vx) \geq 0}$ and thus 
	\begin{equation*}
		f(\vx) - \nu\eps - p(\vx) \geq \Lambda(\vx) - \nu\eps-p(\vx) = g(\vx) \geq \min_{\vx} g(\vx) \geq 0.
	\end{equation*}
	In particular, \labelcref{eq:above-approx-multi-dim} is true for $p$, as well as~\labelcref{eq:integral-tracking-multi-dim} for $v_\vd = 0$, $\vd \in \mathcal{G}$.
	
	Lastly, we prove the feasibility of $p$ for constraints~\labelcref{eq:gradient-bound-a} and \labelcref{eq:gradient-bound-b}. 
	From above, we have the explicit representation of $\alpha_i$ and $\beta_i$, $i \in [n]$. 
	Hence, we directly derive
	\begin{align*}
		\Abs{\frac{\mathrm{d}}{\mathrm{d}x_i} p(\va)} &= \abs{2\alpha_i a_i + \beta_i} = \abs{-2La_i/\Delta + 2Lt_i/\Delta} = \frac{2L}{\Delta} \abs{t_i - a_i} \\
		& = \frac{2L}{\Delta} (t_i - a_i) \leq \frac{2L}{\Delta} (b_i - a_i) \leq \frac{2L}{\Delta}\norm[\infty]{\vb - \va} = C.
	\end{align*}
	This shows that $p$ satisfies~\labelcref{eq:gradient-bound-a}. 
	The case for~\labelcref{eq:gradient-bound-b} follows analogously. 
	This concludes the proof. 
\end{proof}

\subsection{Proof of \Cref{le:p-Lipschitz}}
\label{subsec:proof-le-p-Lipschitz}

\begin{proof}
	Let $i \in [n]$ be arbitrary but fixed.
	Note that $\frac{\mathrm{d}}{\mathrm{d}x_i}p(\vx) = 2 \alpha_i x_i + \beta_i$.
	This is a linear term in $x_i$.
	Since $\vx \in \domain'$, it is $a'_i \leq x_i \leq b'_i$.
	Hence, the bounds on the absolute derivative at $\va'$ and $\vb'$ are equivalent to bounding a linear term in absolute values between $a'_i$ and $b'_i$.
	So, the first claim about the general bound of the derivative on $\domain'$ follows.
	
	Now, let $\vx,\, \vect{y} \in \domain'$, and $g(\lambda) \coloneqq p((1-\lambda)\vx + \lambda \vect{y})$.
	By the chain rule, it follows that $g'(\lambda) = \frac{\mathrm{d}}{\mathrm{d}\lambda} g(\lambda) = \nabla p((1- \lambda) \vx + \lambda \vect{y})\T (\vect{y} - \vx)$.
	By the mean value theorem, there exists $\hat{\lambda} \in [0, 1]$ such that 
	\begin{equation*}
		g(1) - g(0) = g'(\hat{\lambda}) (1 - 0) = g'(\hat{\lambda}),
	\end{equation*}
	and it follows 
	\begin{equation*}
		\abs{p(\vx) - p(\vect{y})} = \abs{g(1) - g(0)} = \abs{g'(\hat{\lambda})} = \abs{\nabla p((1- \hat{\lambda}) \vx + \hat{\lambda} \vect{y})\T (\vect{y} - \vx)}.
	\end{equation*}
	In order to show the Lipschitz continuity of $p$, we now have to find a bound for the upper. 
	This is achieved by an application of Hölder's inequality, in particular, 
	\begin{equation*}
		\abs{p(\vx) - p(\vect{y})} \leq \norm[1]{\nabla p((1- \hat{\lambda}) \vx + \hat{\lambda} \vect{y})} \norm[\infty]{\vect{y} - \vx} \leq n C \norm[\infty]{\vect{y} - \vx},
	\end{equation*}
	and 
	\begin{equation*}
		\abs{p(\vx) - p(\vect{y})} \leq \norm[\infty]{\nabla p((1- \hat{\lambda}) \vx + \hat{\lambda} \vect{y})} \norm[1]{\vect{y} - \vx} \leq C \norm[1]{\vect{y} - \vx},		
	\end{equation*}
	respectively, where we used the bounded partial derivative in each last step.
	This shows the second claim. 
\end{proof}

\subsection{Proof of \Cref{le:g-lower-bound}}
\label{subsec:proof-le-g-lower-bound}

\begin{proof}
	Let $g^* \coloneqq g(\sol) \coloneqq \min_{\vx \in \domain'} g(\vx)$.
	Note that $g^*$ is finite, since $g$ is (Lipschitz) continuous and $\domain'$ is compact.
	Hence, the point $\vx^* \in \domain'$ exists. 
	
	Now, by Caratheodory's Theorem, there exist vertices $\vv^1,\dots,\vv^{n+1}$ of $\domain'$ such that
	\begin{equation*}
		\sol = \sum_{j = 1}^{n + 1} \lambda_j \vv^j \qquad \land \qquad 
		\sum_{j = 1}^{n+1} \lambda_j = 1 \qquad \land \qquad
		\lambda_j \geq 0,\, \text{for all } j \in [n+1].
	\end{equation*}
	Next, let $k \in [n+1]$.
	Then, we can rewrite the first equality from above to
	\begin{equation*}
		\sol - \vv^k = \sum_{\substack{j = 1,\\ j \neq k}}^{n+1} \lambda_j\vv^j + (\lambda_k - 1) \vv^k = \sum_{\substack{j = 1,\\ j \neq k}}^{n+1} \lambda_j\vv^j - \sum_{\substack{j = 1,\\ j \neq k}}^{n+1} \lambda_j\vv^k = \sum_{j = 1}^{n+1} \lambda_j(\vv^j - \vv^k),
	\end{equation*}
	where we used that the sum of $\lambda_j$ is one in the second and $\vv^k - \vv^k = \vect{0}$ in the last step.
	
	By leveraging  $g(\vv^k) \geq 0$ from the assumption, the minimality of $g(\sol)$, and the Lipschitz continuity of $g$ (in this order), we derive
	\begin{align*}
		-g(\sol) &= 0 - g(\sol) \leq g(\vv^k) - g(\sol) = \abs{g(\vv^k) - g(\sol)} \leq L_g \norm[1]{\vv^k - \sol} \\
		&= L_g \Norm[1]{\sum_{j = 1,\, j \neq k}^{n+1} \lambda_j(\vv^j - \vv^k)} = L_g \sum_{i = 1}^n \Abs{\sum_{j = 1}^{n+1} \lambda_j (v_i^j - v_i^k)}.
	\end{align*}
	Note that the $i$th entry of any vertex $\vv$ is either $a'_i$ or $b'_i$.
	This implies that the inner bracket is non-negative if $v_i^k = a'_i$ and non-positive if $v_i^k = b'_i$. 
	Having this in mind and rearranging the sums, we can further rewrite the right-hand side of the upper inequality as 
	\begin{align*}
		L_g \sum_{i \in [n]} \Abs{\sum_{j = 1}^{n+1} \lambda_j (v_i^j - v_i^k)} & = L_g \left[ \sum_{\substack{i \in [n],\\ v_i^k = a'_i}}  \sum_{j = 1}^{n+1} \lambda_j (v_i^j - v_i^k) +  \sum_{\substack{i \in [n],\\ v_i^k = b'_i}}  \sum_{j = 1}^{n+1} \lambda_j (v_i^k - v_i^j) \right] \\
		& = L_g \left[ 
		\sum_{j = 1}^{n+1} \lambda_j \sum_{\substack{i \in [n],\\ v_i^k = a'_i}} (v_i^j - v_i^k) +  \sum_{j = 1}^{n+1} \lambda_j \sum_{\substack{i \in [n],\\ v_i^k = b'_i}} (v_i^k - v_i^j) \right] \\
		& = L_g \sum_{j = 1}^{n+1} \lambda_j \sum_{\substack{i \in [n],\\ v_i^k \neq v_i^j}} (b'_i - a'_i).
	\end{align*}
	Here, $\lambda_j \geq 0$ allows to omit the absolute value in the first equality and one needs to note that the inner brackets are $b'_i - a'_i$ if and only if $v_i^k \neq v_i^j$.
	In conclusion, we receive
	\begin{equation*}
		-g(\sol) \leq L_g \sum_{j = 1}^{n+1} \lambda_j \sum_{\substack{i \in [n],\\ v_i^k \neq v_i^j}} (b'_i - a'_i).
	\end{equation*}
	Summing this over $k$ gives
	\begin{align*}
		-(n+1) g(\sol) &\leq L_g \sum_{k=1}^{n+1} \sum_{j = 1}^{n+1} \lambda_j \sum_{\substack{i \in [n],\\ v_i^k \neq v_i^j}} (b'_i - a'_i) = L_g \sum_{j = 1}^{n+1} \lambda_j \sum_{k = 1}^{n+1} \sum_{\substack{i \in [n],\\ v_i^k \neq v_i^j}} (b'_i - a'_i) \\
		& = L_g \sum_{j = 1}^{n+1} \lambda_j \sum_{\substack{k = 1,\\ k \neq j}}^{n+1} \sum_{\substack{i \in [n],\\ v_i^k \neq v_i^j}} (b'_i - a'_i) \leq 
		L_g \sum_{j = 1}^{n+1} \lambda_j n \sum_{i = 1}^n (b'_i - a'_i) \\
		& = L_g n \sum_{j = 1}^{n+1} \lambda_j \norm[1]{\vb' - \va'} = L_g n \norm[1]{\vb' - \va'}.
	\end{align*}
	A division by $-(n+1)$ then gives the desired bound and finishes the proof. 
\end{proof}

\subsection{Proof of \Cref{le:g-upper-bound}}
\label{subsec:proof-le-g-upper-bound}

\begin{proof}
	Let $g^* \coloneqq g(\sol) \coloneqq \max_{\vx \in \domain'} g(\vx)$.
	Again, note that $g^*$ is finite and such a point $\vx^*$ exists due to the (Lipschitz) continuity of $g$ and the compactness of $\domain'$.
	
	Since $g$ is Lipschitz, we have 
	\begin{equation*}
		g^* - g(\vx) = \abs{g(\vx^*) - g(\vx)} \leq L_g \norm[1]{\vx - \sol},
	\end{equation*}
	which is equivalent to 
	\begin{equation}
		\label{eq:def-hat-function}
		g(\vx) \geq g^* - L_g \norm[1]{\vx - \sol} \eqqcolon \Lambda (\vx),
	\end{equation}
	for all $\vx \in \domain'$.
	Now, with the assumption and by writing $V \coloneqq \mathrm{vol}(\domain')$, we derive
	\begin{align*}
		0 &\geq \int_{\domain'} g(\vx)\mathrm{d}\vx 
		\geq \int_{\domain'} \Lambda(\vx) \mathrm{d}\vx 
		= \int_{\domain'} g^* - L_g\norm[1]{\vx - \sol} \mathrm{d}\vx  \\
		& = Vg^* - L_g \int_{\domain'} \norm[1]{\vx - \sol}\mathrm{d}\vx 
		= Vg^* - L_g \int_{[\va' - \sol, \vb'- \sol]} \norm[1]{\vx}\mathrm{d}\vx  \\ 
		& = \begin{aligned}[t] 
			& Vg^* - \\
			&  L_g \left[ \frac12 \mathrm{vol}([\va'-\sol, \vb'-\sol]) \sum_{i = 1}^n \left(b'_i - x_i^* - (a'_i - x_i^*) + 2 \frac{(b'_i- x_i^*)(a'_i - x_i^*)}{b'_i - x_i^* - (a'_i - x_i^*)}\right)\right] 
		\end{aligned} \\
		& = Vg^* - L_g \frac12 \mathrm{vol}([\va', \vb'])\left[ \sum_{i = 1}^n (b'_i - a'_i ) + 2 \frac{(b'_i- x_i^*)(a'_i - x_i^*)}{b'_i - a'_i}\right] \\
		& = Vg^* - \frac{L_gV}{2} \left[\norm[1]{\vb' - \va'} + 2 \sum_{i = 1}^n \frac{(b'_i- x_i^*)(a'_i - x_i^*)}{b'_i - a'_i}\right].
	\end{align*}
	A rearrangement gives 
	\begin{equation}
		\label{eq:first-bound-gstar}
		g^* \leq \frac{L_g}{2} \norm[1]{\vb' - \va'} + L_g \sum_{i = 1}^n \frac{(b'_i- x_i^*)(a'_i - x_i^*)}{b'_i - a'_i},
	\end{equation}
	where it remains to investigate the sum.
	We note that $0 < b'_i - a'_i \leq \Delta_{\max}$ for all $i \in [n]$ by definition.
	In addition, each vertex $\vv$ of $\domain'$ is either $a'_i$ or $b'_i$ in its $i$th entry.
	Hence, there exists a vertex $\vect{v}$ of $\mathcal{D}'$ which attains the minimum in the inequality
	\begin{equation*}
		(b'_i - x^*_i)(x^*_i - a'_i) \geq \min\{(b'_i - x^*_i),(x^*_i - a'_i)\}^2 = (v_i - x_i^*)^2,
	\end{equation*}
	for all $i \in [n]$.
	Using this $\vv$, we can bound the sum in~\eqref{eq:first-bound-gstar} as 
	\begin{equation*}
		\sum_{i = 1}^n \frac{(b'_i- x_i^*)(a'_i - x_i^*)}{b'_i - a'_i } 
		\leq -\frac{1}{\Delta_{\max}} \sum_{i = 1}^n (v_i - x_i^*)^2 
		= -\frac{1}{\Delta_{\max}} \norm[2]{\vv - \sol}^2.
	\end{equation*}
	As \eqref{eq:def-hat-function} is especially true for $\vv$, the assumption to be non-positive at vertices gives $0 \geq g(\vv) \geq g^* - L_g \norm[1]{\vv - \sol}$ which is equivalent to $\norm[1]{\vv - \sol} \geq g^*/L_g$. 
	Combined with the well-known estimate ${\sqrt{n}\norm[2]{\vx} \geq \norm[1]{\vx}}$, which is equivalent to $\norm[2]{\vx}^2 \geq \frac1n\norm[1]{\vx}^2$, we receive
	\begin{equation*}
		-\frac{1}{\Delta_{\max}} \norm[2]{\vv - \sol}^2 
		\leq - \frac{1}{\Delta_{\max} n} \norm[1]{\vv - \sol}^2 
		\leq -\frac{(g^*)^2}{\Delta_{\max} n L_g^2}.
	\end{equation*}
	In summary, \eqref{eq:first-bound-gstar} breaks down to 
	\begin{equation*}
		g^* \leq \frac{L_g}{2} \norm[1]{\vb' - \va'} - \frac{(g^*)^2}{\Delta_{\max} n L_g},
	\end{equation*}
	which can be rearranged to
	\begin{equation*}
		\frac12 (g^*)^2 + \frac12 \Delta_{\max} n L_g g^*- \frac14 \Delta_{\max} L_g^2 n \norm[1]{\vb' - \va'} \leq 0.
	\end{equation*}
	Solving this quadratic (in)equality for the positive solution leads to
	\begin{align*}
		g^* & \leq -\frac12 \Delta_{\max} n L_g + \sqrt{\frac14 \Delta_{\max}^2 n^2 L_g^2 + \frac12 \Delta_{\max} L_g^2 n \norm[1]{\vb' - \va'}} \\
		& \leq -\frac12 \Delta_{\max} n L_g + \sqrt{ \frac14 \Delta_{\max}^2 n^2 L_g^2 + \frac12 \Delta_{\max}^2 n^2 L_g^2} \\
		& = -\frac12 \Delta_{\max} n L_g + \sqrt{\frac34 \Delta_{\max}^2 n^2 L_g^2} = -\frac12 \Delta_{\max} n L_g + \frac{\sqrt{3}}{2} \Delta_{\max} n L_g \\
		& = \frac{\sqrt{3} - 1}{2} \Delta_{\max} n L_g,
	\end{align*}
	where we used in the second inequality that $\norm[1]{\vect{b}'- \vect{a}'} \leq n \Delta_{\max}$.
	This gives the desired bound.
\end{proof}

\section{Parabolic Approximations -- Remaining Results \& Plots}

\subsection{Definition of the Zigzag Function}
\label{subsec:zigzag-def}

The zigzag function mentioned and investigated in~\Cref{subsec:comp-para-approx} is defined by the points in~\Cref{tab:zigzag-points}, rounded to two digits for clear presentation.
All $x$-points are sampled randomly from the uniform distribution on the interval from the previous point plus one. 
The function value of $-5.00$ is sampled uniformly on $[-1, 1]$ and every consecutive one on an interval such that the maximal slope of the zigzag function is one.
The random seed in the \software{Python} implementation is two, see the repository.

\begin{table}[h]
	\centering
	\begin{tabular}{llcccccccccc}
		\toprule
		$x$ & & -5.00 & -4.96 & -4.52 &-4.19 & -3.56 & -3.29 & -2.76 & -2.24 & -1.45 & -0.95 \\
		$f(x)$ & & -0.13 & -0.12 & -0.19 & -0.39 & -0.64 & -0.58 & -0.97 & -1.29 & -0.74 & -0.39 \\
		\bottomrule
	\end{tabular}
	\begin{tabular}{llcccccccccc}
		\toprule
		$x$ & & -0.86 & -0.79 & -0.68 & -0.08 & 0.04 & 0.39 & 0.60 & 1.09 & 1.48 & 2.07 \\
		$f(x)$ & & -0.39 & -0.40 & -0.48 & -0.81 & -0.87 & -0.90 & -0.84 & -0.83 & -0.60 & -1.00 \\
		\bottomrule
	\end{tabular}
	\begin{tabular}{llcccccc}
		\toprule
		$x$ & & 2.77 & 3.28 &  3.62 & 4.06 & 4.84 & 5.00 \\
		$f(x)$ &  & -0.34 & 0.05 & 0.10 & 0.04 & 0.10 & 0.11 \\
		\bottomrule
	\end{tabular}
	\caption{Sample points defining the zigzag function.}
	\label{tab:zigzag-points}
\end{table}

\subsection{Number of Paraboloids for Approximating the Zigzag Function}
\label{subsec:approx-zigzag}
\
\begin{table}[H]
	\centering
	\begin{tabular}{ccccccccccc}
		\toprule
		& & \multicolumn{9}{c}{$\eps$} \\ \cline{3-11}
		& & \multicolumn{9}{c}{\vspace{-0.28cm}} \\
		$\domain$ & & $10^{0}$ & $10^{-1}$ & $10^{-2}$ & $10^{-3}$  & & $10^{0}$ & $10^{-1}$ & $10^{-2}$ & $10^{-3}$ \\ \cline{3-6}\cline{8-11}
		\\[-1em] 
		$[-5, -1]$ &  & 2 & 7 & - & - &  & 1 & 4 & - & - \\ 
		\\[-1em] 
		$[-2, \phantom{-}2]$ &  & 2 & 6 & - & - &  & 1 & 3 & 32 & - \\ 
		\\[-1em] 
		$[\phantom{-}1, \phantom{-}5]$ &  & 2 & 7 & - & - &  & 1 & 5 & - & - \\ 
		\\[-1em] 
		$[-5, \phantom{-}0]$ &  & 3 & 9 & - & - &  & 2 & 5 & - & - \\ 
		\\[-1em] 
		$[\phantom{-}0, \phantom{-}5]$ &  & 3 & 8 & - & - &  & 1 & 6 & - & - \\ 
		\\[-1em] 
		$[-5, \phantom{-}5]$ &  & 5 & - & - & - &  & 2 & 10 & - & - \\ 
		& & \multicolumn{4}{c}{\texttt{exact-MIP}} & & \multicolumn{4}{c}{\texttt{practical-MIP}} \\
		\bottomrule
	\end{tabular}
	\caption{Number of paraboloids to approximate the zigzag function from above.}
	\label{tab:number-paraboloids-zigzag-above}
\end{table}

\subsection{Number of Paraboloids for Approximating $x^3$}
\label{subsec:approx-x3}
\
\begin{table}[H]
	\centering
	\begin{tabular}{ccccccccccc}
		\toprule
		& & \multicolumn{9}{c}{$\eps$} \\ \cline{3-11}
		& & \multicolumn{9}{c}{\vspace{-0.28cm}} \\
		$\domain$ & & $10^{0}$ & $10^{-1}$ & $10^{-2}$ & $10^{-3}$  & & $10^{0}$ & $10^{-1}$ & $10^{-2}$ & $10^{-3}$ \\ \cline{3-6}\cline{8-11}
		\\[-1em] 
		$[-2, \phantom{-}2]$ &  & 5 & - & - & - &  & 3 & 11 & 169 & - \\ 
		\\[-1em] 
		$[-5, \phantom{-}0]$ &  & - & - & - & - &  & 5 & 16 & - & - \\ 
		\\[-1em] 
		$[\phantom{-}0, \phantom{-}5]$ &  & - & - & - & - &  & 4 & 17 & - & - \\ 
		\\[-1em] 
		$[-5, \phantom{-}2]$ &  & - & - & - & - &  & - & - & - & - \\ 
		\\[-1em] 
		$[-2, \phantom{-}5]$ &  & - & - & - & - &  & - & - & - & - \\ 
		\\[-1em] 
		$[-5, \phantom{-}5]$ &  & - & - & - & - &  & - & - & - & - \\ 
		& & \multicolumn{4}{c}{\texttt{exact-MIP}} & & \multicolumn{4}{c}{\texttt{practical-MIP}} \\
		\bottomrule
	\end{tabular}
	\caption{Number of paraboloids to approximate $x^3$ from above.}
	\label{tab:number-paraboloids-x3-above}
\end{table}
\begin{table}[H]
	\centering
	\begin{tabular}{ccccccccccc}
		\toprule
		& & \multicolumn{9}{c}{$\eps$} \\ \cline{3-11}
		& & \multicolumn{9}{c}{\vspace{-0.28cm}} \\
		$\domain$ & & $10^{0}$ & $10^{-1}$ & $10^{-2}$ & $10^{-3}$  & & $10^{0}$ & $10^{-1}$ & $10^{-2}$ & $10^{-3}$ \\ \cline{3-6}\cline{8-11}
		\\[-1em] 
		$[-2, \phantom{-}2]$ &  & 5 & - & - & - &  & 3 & 12 & 97 & - \\ 
		\\[-1em] 
		$[-5, \phantom{-}0]$ &  & - & - & - & - &  & 5 & 17 & - & - \\ 
		\\[-1em] 
		$[\phantom{-}0, \phantom{-}5]$ &  & - & - & - & - &  & 5 & 17 & - & - \\ 
		\\[-1em] 
		$[-5, \phantom{-}2]$ &  & - & - & - & - &  & - & - & - & - \\ 
		\\[-1em] 
		$[-2, \phantom{-}5]$ &  & - & - & - & - &  & - & - & - & - \\ 
		\\[-1em] 
		$[-5, \phantom{-}5]$ &  & - & - & - & - &  & - & - & - & - \\ 
		& & \multicolumn{4}{c}{\texttt{exact-MIP}} & & \multicolumn{4}{c}{\texttt{practical-MIP}} \\
		\bottomrule
	\end{tabular}
	\caption{Number of paraboloids to approximate $x^3$ from below.}
	\label{tab:number-paraboloids-x3-below}
\end{table}

\subsection{Plots for the Parabolic Approximations}
\label{subsec:plots}
For the zigzag function, \Cref{fig:zigzag-approx-below} visualizes the approximation results on $[-5, 5]$ from below for $\eps = 10^{-1}$. 
In~\Cref{fig:approx-zigzag-below-e-0}, we visualize the latter for $\eps = 10^0$ and in~\Cref{fig:approx-zigzag-above-e-0} and~\Cref{fig:approx-zigzag-above-e-1} the situation from above.

\begin{figure}[H]
	\caption{Approximation with $\eps = 10^0$ of the zigzag function on $[-5, 5]$ from below by 2 paraboloids.}
	\label{fig:approx-zigzag-below-e-0}
	\centering 
	\begin{tikzpicture}
		\begin{axis}[
			axis lines=middle,
			xlabel={$x$},
			ylabel={$f(x)$},
			grid=none,
			width=12cm,
			height=8cm,
			xtick={-5,-4,-3,-2,-1,0,1,2,3,4,5},
			ytick={-1.5,-1,-0.5,0,0.5,1},
			clip mode=individual,
			ymin=-1.8, ymax=0.6
			]
			\addplot[color=black, line width=2pt,mark size=0pt] coordinates {
				(-5,-0.12801019571599248)
				(-4.964333030490388,-0.12446757554744907)
				(-4.523363861798294,-0.1946982637766576)
				(-4.186332389004458,-0.3937836755004518)
				(-3.563254132317301,-0.6434453088250474)
				(-3.289095129965463,-0.5770254473388788)
				(-2.7552444566311944,-0.9671849317937837)
				(-2.2368021165781053,-1.2943844007567447)
				(-1.4493203202395986,-0.7368862021423581)
				(-0.9500258512314901,-0.3908137365219214)
				(-0.8611768289925197,-0.38988151656613956)
				(-0.7865431896495103,-0.4006105011239624)
				(-0.6809775831455038,-0.4793286685566136)
				(-0.080199727256694,-0.808540514189117)
				(0.035676500210189144,-0.8733602373302513)
				(0.3920045223634554,-0.8963166811449923)
				(0.601730316365771,-0.8374228572870663)
				(1.0899694535620053,-0.832309313959877)
				(1.482993178169379,-0.6014963419021636)
				(2.0671973152584076,-0.9960694536135861)
				(2.770942138399492,-0.34221861821876076)
				(3.2759504159580066,0.05120309491708941)
				(3.6241479321024053,0.09796193184439114)
				(4.0574184357655545,0.043150841431488146)
				(4.8362120289126835,0.09860744585003525)
				(5,0.11308897556235686)
			};
			\foreach \qcoeff/\lcoeff/\constcoeff in {
				0.04162/0.08301/-1.345523,
				-2.76964/-5.5995/-3.30654
			} {	
				\addplot[domain=-5:5, smooth, variable=\x, gray]
				({\x}, {\qcoeff*\x*\x + \lcoeff*\x + \constcoeff});
			}
		\end{axis}
	\end{tikzpicture}
\end{figure}
\begin{figure}[H]
	\caption{Approximation with $\eps = 10^0$ of the zigzag function on $[-5, 5]$ from above by 2 paraboloids.}
	\label{fig:approx-zigzag-above-e-0}
	\centering 
	\begin{tikzpicture}
		\begin{axis}[
			axis lines=middle,
			xlabel={$x$},
			ylabel={$f(x)$},
			grid=none,
			width=12cm,
			height=8cm,
			xtick={-5,-4,-3,-2,-1,0,1,2,3,4,5},
			ytick={-1.5,-1,-0.5,0,0.5,1},
			clip mode=individual,
			ymin=-1.8, ymax=0.6
			]
			\addplot[color=black, line width=2pt,mark size=0pt] coordinates {
				(-5,-0.12801019571599248)
				(-4.964333030490388,-0.12446757554744907)
				(-4.523363861798294,-0.1946982637766576)
				(-4.186332389004458,-0.3937836755004518)
				(-3.563254132317301,-0.6434453088250474)
				(-3.289095129965463,-0.5770254473388788)
				(-2.7552444566311944,-0.9671849317937837)
				(-2.2368021165781053,-1.2943844007567447)
				(-1.4493203202395986,-0.7368862021423581)
				(-0.9500258512314901,-0.3908137365219214)
				(-0.8611768289925197,-0.38988151656613956)
				(-0.7865431896495103,-0.4006105011239624)
				(-0.6809775831455038,-0.4793286685566136)
				(-0.080199727256694,-0.808540514189117)
				(0.035676500210189144,-0.8733602373302513)
				(0.3920045223634554,-0.8963166811449923)
				(0.601730316365771,-0.8374228572870663)
				(1.0899694535620053,-0.832309313959877)
				(1.482993178169379,-0.6014963419021636)
				(2.0671973152584076,-0.9960694536135861)
				(2.770942138399492,-0.34221861821876076)
				(3.2759504159580066,0.05120309491708941)
				(3.6241479321024053,0.09796193184439114)
				(4.0574184357655545,0.043150841431488146)
				(4.8362120289126835,0.09860744585003525)
				(5,0.11308897556235686)
			};
			\foreach \qcoeff/\lcoeff/\constcoeff in {
				0.16655/0.74766/0.169164,
				0.04529/-0.00215/-0.42532
			} {	
				\addplot[domain=-5:5, smooth, variable=\x, gray]
				({\x}, {\qcoeff*\x*\x + \lcoeff*\x + \constcoeff});
			}
		\end{axis}
	\end{tikzpicture}
\end{figure}

\begin{figure}[H]
	\caption{Approximation with $\eps = 10^{-1}$ of the zigzag function on $[-5, 5]$ from above by 10 paraboloids.}
	\label{fig:approx-zigzag-above-e-1}
	\centering 
	\begin{tikzpicture}
		\begin{axis}[
			axis lines=middle,
			xlabel={$x$},
			ylabel={$f(x)$},
			grid=none,
			width=12cm,
			height=8cm,
			xtick={-5,-4,-3,-2,-1,0,1,2,3,4,5},
			ytick={-1.5,-1,-0.5,0,0.5,1},
			clip mode=individual,
			ymin=-1.8, ymax=0.6
			]
			\addplot[color=black, line width=2pt,mark size=0pt] coordinates {
				(-5,-0.12801019571599248)
				(-4.964333030490388,-0.12446757554744907)
				(-4.523363861798294,-0.1946982637766576)
				(-4.186332389004458,-0.3937836755004518)
				(-3.563254132317301,-0.6434453088250474)
				(-3.289095129965463,-0.5770254473388788)
				(-2.7552444566311944,-0.9671849317937837)
				(-2.2368021165781053,-1.2943844007567447)
				(-1.4493203202395986,-0.7368862021423581)
				(-0.9500258512314901,-0.3908137365219214)
				(-0.8611768289925197,-0.38988151656613956)
				(-0.7865431896495103,-0.4006105011239624)
				(-0.6809775831455038,-0.4793286685566136)
				(-0.080199727256694,-0.808540514189117)
				(0.035676500210189144,-0.8733602373302513)
				(0.3920045223634554,-0.8963166811449923)
				(0.601730316365771,-0.8374228572870663)
				(1.0899694535620053,-0.832309313959877)
				(1.482993178169379,-0.6014963419021636)
				(2.0671973152584076,-0.9960694536135861)
				(2.770942138399492,-0.34221861821876076)
				(3.2759504159580066,0.05120309491708941)
				(3.6241479321024053,0.09796193184439114)
				(4.0574184357655545,0.043150841431488146)
				(4.8362120289126835,0.09860744585003525)
				(5,0.11308897556235686)
			};
			\foreach \qcoeff/\lcoeff/\constcoeff in {
				-0.00116/0.0244/0.025251,
				1.92379/8.69682/8.601094,
				0.55121/-2.12036/1.330725,
				0.42472/-0.33794/-0.787855,
				0.70845/-1.09951/-0.40895,
				0.11399/-0.17167/-0.597605,
				0.62702/2.82788/2.081011,
				0.31329/2.09793/2.93405,
				0.32132/1.47348/0.793305,
				2.04038/-8.2865/7.497918
			} {	
				\addplot[domain=-5:5, smooth, variable=\x, gray]
				({\x}, {\qcoeff*\x*\x + \lcoeff*\x + \constcoeff});
			}
		\end{axis}
	\end{tikzpicture}
\end{figure}

For the sine function, we visualize its approximations on $[-\pi/2, 3\pi/2]$ from above for $\eps = 10^0, 10^{-1}, 10^{-2}$ in~\Cref{fig:approx-sin-above-e-0,fig:approx-sin-above-e-1,fig:approx-sin-above-e-2}, respectively, and the approximations from below for $\eps = 10^0, 10^{-1}$ in~\Cref{fig:approx-sin-below-e-0} and~\Cref{fig:approx-sin-below-e-1}. 
The case $\eps = 10^{-2}$ for the latter can be found in~\Cref{fig:approx-sin-below-e-2}.

\begin{figure}[H]
	\caption{Approximation  with $\eps = 10^{0}$ of $\sin$ on $[-\pi/2, 3\pi/2]$ from above by 1 paraboloid.}
	\label{fig:approx-sin-above-e-0}
	\centering
	
	\begin{tikzpicture}[scale=1.8]
		\clip (-1.57,-1.2) rectangle (4.71,1.4);
		\ 
		\draw[->] (-1.57,0) -- (4.71,0) node[right] {$x$};
		\draw[->] (0,-1.2) -- (0,1.4) node[above] {$y$};
		\foreach \x in {-1.5,-1.0,...,4.5}
		\draw (\x,-0.05) -- (\x,0.05);
		\foreach \y in {-1.0,-0.5,...,1.0}
		\draw (-0.05,\y) -- (0.05,\y);
		\foreach \x/\label in {-1.5,-1.0,-0.5,0.5,1.0,1.5,2.0,2.5,3.0,3.5,4.0,4.5}
		\draw (\x,0) node[below] {$\label$};
		\foreach \y in {-1.0,-0.5,0.5,1.0}
		\draw (0,\y) node[left] {$\y$};
		\draw[domain=-1.57:4.71,smooth,variable=\x,black, line width=2pt] plot ({\x},{sin(deg(\x))});
		\def\quadraticcoeffs{{-0.21837,0}}
		\def\linearcoeffs{{0.68602,0}}
		\def\constantcoeffs{{0.616465,0}}
		\foreach \i in {0,...,0}{
			\draw[domain=-1.57:4.71,smooth,variable=\x,gray] plot ({\x},{\quadraticcoeffs[\i]*\x*\x + \linearcoeffs[\i]*\x + \constantcoeffs[\i]});
		}
	\end{tikzpicture}
\end{figure}
\begin{figure}[H]
	\caption{Approximation  with $\eps = 10^{-1}$ of $\sin$ on $[-\pi/2, 3\pi/2]$ from above by 5 paraboloids.}
	\label{fig:approx-sin-above-e-1}
	\centering
	
	\begin{tikzpicture}[scale=1.8]
		\clip (-1.57,-1.2) rectangle (4.71,1.4);
		\ 
		\draw[->] (-1.57,0) -- (4.71,0) node[right] {$x$};
		\draw[->] (0,-1.2) -- (0,1.4) node[above] {$y$};
		\foreach \x in {-1.5,-1.0,...,4.5}
		\draw (\x,-0.05) -- (\x,0.05);
		\foreach \y in {-1.0,-0.5,...,1.0}
		\draw (-0.05,\y) -- (0.05,\y);
		\foreach \x/\label in {-1.5,-1.0,-0.5,0.5,1.0,1.5,2.0,2.5,3.0,3.5,4.0,4.5}
		\draw (\x,0) node[below] {$\label$};
		\foreach \y in {-1.0,-0.5,0.5,1.0}
		\draw (0,\y) node[left] {$\y$};
		\draw[domain=-1.57:4.71,smooth,variable=\x,black, line width=2pt] plot ({\x},{sin(deg(\x))});
		\def\quadraticcoeffs{{0.3032,0.07435,-0.16698,0.02547,0.24552}}
		\def\linearcoeffs{{1.09246,0.78482,0.58126,-0.87741,-2.56527}}
		\def\constantcoeffs{{0.00779,0.068695,0.50138,2.63476,5.636388}}
		\foreach \i in {0,...,4}{
			\draw[domain=-1.57:4.71,smooth,variable=\x,gray] plot ({\x},{\quadraticcoeffs[\i]*\x*\x + \linearcoeffs[\i]*\x + \constantcoeffs[\i]});
		}
	\end{tikzpicture}
\end{figure}

\begin{figure}[H]
	\caption{Approximation  with $\eps = 10^{-2}$ of $\sin$ on $[-\pi/2, 3\pi/2]$ from above by 47 paraboloids.}
	\label{fig:approx-sin-above-e-2}
	\centering
	
	\begin{tikzpicture}[scale=1.8]
		\clip (-1.57,-1.2) rectangle (4.71,1.4);
		\ 
		\draw[->] (-1.57,0) -- (4.71,0) node[right] {$x$};
		\draw[->] (0,-1.2) -- (0,1.4) node[above] {$y$};
		\foreach \x in {-1.5,-1.0,...,4.5}
		\draw (\x,-0.05) -- (\x,0.05);
		\foreach \y in {-1.0,-0.5,...,1.0}
		\draw (-0.05,\y) -- (0.05,\y);
		\foreach \x/\label in {-1.5,-1.0,-0.5,0.5,1.0,1.5,2.0,2.5,3.0,3.5,4.0,4.5}
		\draw (\x,0) node[below] {$\label$};
		\foreach \y in {-1.0,-0.5,0.5,1.0}
		\draw (0,\y) node[left] {$\y$};
		\draw[domain=-1.57:4.71,smooth,variable=\x,black, line width=2pt] plot ({\x},{sin(deg(\x))});
		\def\quadraticcoeffs{{-0.20231,0.06736,0.41721,-0.09784,-0.07369,0.9238,0.36683,-0.0147,0.20208,0.46708,0.19835,0.42595,0.41707,3.67596,0.39481,2.70845,-0.15427,0.24152,0.12879,0.2046,-0.12062,0.48584,0.09427,-0.01203,-0.00435,-0.15848,0.28305,0.2497,0.40011,0.39954,6.1725,3.46099,0.40637,0.48458,0.35661,0.40294,-0.0021,0.49745,0.10235,3.00306,0.02891,0.10276,8.82625,-0.14572,0.05131,0.2581,0.31703}}
		\def\linearcoeffs{{0.63455,0.78677,1.35612,0.66207,-0.21129,-2.19396,1.24537,0.71307,0.95549,-4.40224,-2.19631,1.25937,-3.97638,3.58895,-3.62281,-24.40265,0.32695,1.01618,0.85638,-2.24478,0.65256,-4.55465,-1.40808,-0.63955,-0.69385,0.35431,1.08585,-2.59834,1.3113,-3.82649,-10.77263,-26.25302,-3.75062,-4.46252,-3.46542,1.32365,-0.70973,1.38906,-1.46801,6.32989,0.74945,0.82539,-24.35748,0.27106,-1.09302,1.04326,-3.13945}}
		\def\constantcoeffs{{0.502433,0.069652,0.100778,0.281389,1.632087,2.255116,0.051555,0.15636,0.002269,9.379322,4.945187,0.046185,8.47654,0.463347,7.498757,53.992102,0.885174,0.000286,0.027432,5.034805,0.322661,9.688871,3.542005,2.280955,2.366291,0.849657,0.007254,5.699399,0.078173,8.159481,5.482035,49.228713,7.799316,9.348557,7.411329,0.084981,2.391159,0.090937,3.644989,2.489892,0.105902,0.042972,17.787401,0.958608,3.011315,0.001926,6.754149}}
		Number of paraboloids: 47
		\foreach \i in {0,...,46}{
			\draw[domain=-1.57:4.71,smooth,variable=\x,gray] plot ({\x},{\quadraticcoeffs[\i]*\x*\x + \linearcoeffs[\i]*\x + \constantcoeffs[\i]});
		}
	\end{tikzpicture}
\end{figure}

\begin{figure}[H]
	\caption{Approximation  with $\eps = 10^{-0}$ of $\sin$ on $[-\pi/2, 3\pi/2]$ from below by 1 paraboloid.}
	\label{fig:approx-sin-below-e-0}
	\centering
	
	\begin{tikzpicture}[scale=1.8]
		\clip (-1.57,-1.2) rectangle (4.71,1.4);
		\ 
		\draw[->] (-1.57,0) -- (4.71,0) node[right] {$x$};
		\draw[->] (0,-1.2) -- (0,1.4) node[above] {$y$};
		\foreach \x in {-1.5,-1.0,...,4.5}
		\draw (\x,-0.05) -- (\x,0.05);
		\foreach \y in {-1.0,-0.5,...,1.0}
		\draw (-0.05,\y) -- (0.05,\y);
		\foreach \x/\label in {-1.5,-1.0,-0.5,0.5,1.0,1.5,2.0,2.5,3.0,3.5,4.0,4.5}
		\draw (\x,0) node[below] {$\label$};
		\foreach \y in {-1.0,-0.5,0.5,1.0}
		\draw (0,\y) node[left] {$\y$};
		\draw[domain=-1.57:4.71,smooth,variable=\x,black, line width=2pt] plot ({\x},{sin(deg(\x))});
		\def\quadraticcoeffs{{-0.21837,0}}
		\def\linearcoeffs{{0.68602,0}}
		\def\constantcoeffs{{-0.08222,0}}
		\foreach \i in {0,...,0}{
			\draw[domain=-1.57:4.71,smooth,variable=\x,gray] plot ({\x},{\quadraticcoeffs[\i]*\x*\x + \linearcoeffs[\i]*\x + \constantcoeffs[\i]});
		}
	\end{tikzpicture}
\end{figure}

\begin{figure}[H]
	\caption{Approximation  with $\eps = 10^{-1}$ of $\sin$ on $[-\pi/2, 3\pi/2]$ from below by 3 paraboloids.}
	\label{fig:approx-sin-below-e-1}
	\centering
	
	\begin{tikzpicture}[scale=1.8]
		\clip (-1.57,-1.2) rectangle (4.71,1.4);
		\ 
		\draw[->] (-1.57,0) -- (4.71,0) node[right] {$x$};
		\draw[->] (0,-1.2) -- (0,1.4) node[above] {$y$};
		\foreach \x in {-1.5,-1.0,...,4.5}
		\draw (\x,-0.05) -- (\x,0.05);
		\foreach \y in {-1.0,-0.5,...,1.0}
		\draw (-0.05,\y) -- (0.05,\y);
		\foreach \x/\label in {-1.5,-1.0,-0.5,0.5,1.0,1.5,2.0,2.5,3.0,3.5,4.0,4.5}
		\draw (\x,0) node[below] {$\label$};
		\foreach \y in {-1.0,-0.5,0.5,1.0}
		\draw (0,\y) node[left] {$\y$};
		\draw[domain=-1.57:4.71,smooth,variable=\x,black, line width=2pt] plot ({\x},{sin(deg(\x))});
		\def\quadraticcoeffs{{-0.05375,-0.20301,-0.39804}}
		\def\linearcoeffs{{0.16887,0.63779,1.25049}}
		\def\constantcoeffs{{-0.654036,-0.110597,-0.047268}}
		\foreach \i in {0,...,2}{
			\draw[domain=-1.57:4.71,smooth,variable=\x,gray] plot ({\x},{\quadraticcoeffs[\i]*\x*\x + \linearcoeffs[\i]*\x + \constantcoeffs[\i]});
		}
	\end{tikzpicture}
\end{figure}

The approximations for the $\exp$ function from below on $[-2, 2]$ are summarized in~\Cref{fig:approx-exp-below}, whereas~\Cref{fig:approx-exp-above} shows the case when approximating from above.

\begin{figure}[H]
	\caption{Approximation  with $\eps = 10^0, 10^{-1}, 10^{-2}$ of $\exp$ on $[-2, 2]$ from below by 2, 4, 20 paraboloids, respectively.}
	\label{fig:approx-exp-below}
	\centering
	\begin{subfigure}{0.32\textwidth}
		\centering
		\begin{tikzpicture}[scale=1.0]
				\clip (-2.1, -1.0) rectangle (2.1, 7.4);
				\ 
				\draw[->] (-2,0) -- (2,0) node[right] {$x$};
				\draw[->] (0,-0.6) -- (0,7.4) node[above] {$y$};
				\foreach \x in {-2.0,-1.5,...,2.0}
				\draw (\x,-0.05) -- (\x,0.05);
				\foreach \y in {0.5,1.0,...,7.0}
				\draw (-0.05,\y) -- (0.05,\y);
				\foreach \x/\label in {-2,-1,1,2}
				\draw (\x,0) node[below] {$\label$};
				\foreach \y in {1,...,7}
				\draw (0,\y) node[left] {$\y$};
				\draw[domain=-2.0:2.0,smooth,variable=\x,black, line width=2pt] plot ({\x},{exp(\x)});
				\def\quadraticcoeffs{{0.57873,-3.77815}}
				\def\linearcoeffs{{1.53094,26.97593}}
				\def\constantcoeffs{{0.608322,-31.450205}}
				\foreach \i in {0,...,1}{
						\draw[domain=-2.0:2.0,smooth,variable=\x,gray] plot ({\x},{\quadraticcoeffs[\i]*\x*\x + \linearcoeffs[\i]*\x + \constantcoeffs[\i]});
					}
			\end{tikzpicture}
	\end{subfigure}
	\hfill
	\begin{subfigure}{0.32\textwidth}
		\centering
		\begin{tikzpicture}[scale=1.0]
			\clip (-2.1, -1.0) rectangle (2.1, 7.4);
			\ 
			\draw[->] (-2,0) -- (2,0) node[right] {$x$};
			\draw[->] (0,-0.6) -- (0,7.4) node[above] {$y$};
			\foreach \x in {-2.0,-1.5,...,2.0}
			\draw (\x,-0.05) -- (\x,0.05);
			\foreach \y in {0.5,1.0,...,7.0}
			\draw (-0.05,\y) -- (0.05,\y);
			\foreach \x/\label in {-2,-1,1,2}
			\draw (\x,0) node[below] {$\label$};
			\foreach \y in {1,...,7}
			\draw (0,\y) node[left] {$\y$};
			\draw[domain=-2.0:2.0,smooth,variable=\x,black, line width=2pt] plot ({\x},{exp(\x)});
			\def\quadraticcoeffs{{0.25121,0.50878,0.8351,1.19818}}
			\def\linearcoeffs{{0.93225,1.35704,1.65423,1.79718}}
			\def\constantcoeffs{{0.994872,0.813026,0.101645,-1.067885}}
			\foreach \i in {0,...,3}{
				\draw[domain=-2.0:2.0,smooth,variable=\x,gray] plot ({\x},{\quadraticcoeffs[\i]*\x*\x + \linearcoeffs[\i]*\x + \constantcoeffs[\i]});
			}
		\end{tikzpicture}
	\end{subfigure}
	\hfill
	\begin{subfigure}{0.32\textwidth}
		\centering
		\begin{tikzpicture}[scale=1.0]
			\clip (-2.1, -1.0) rectangle (2.1, 7.4);
			\ 
			\draw[->] (-2,0) -- (2,0) node[right] {$x$};
			\draw[->] (0,-0.6) -- (0,7.4) node[above] {$y$};
			\foreach \x in {-2.0,-1.5,...,2.0}
			\draw (\x,-0.05) -- (\x,0.05);
			\foreach \y in {0.5,1.0,...,7.0}
			\draw (-0.05,\y) -- (0.05,\y);
			\foreach \x/\label in {-2,-1,1,2}
			\draw (\x,0) node[below] {$\label$};
			\foreach \y in {1,...,7}
			\draw (0,\y) node[left] {$\y$};
			\draw[domain=-2.0:2.0,smooth,variable=\x,black, line width=2pt] plot ({\x},{exp(\x)});
			\def\quadraticcoeffs{{0.32149,0.49213,-2.35214,0.25018,0.78822,1.14861,0.40455,1.02457,1.27344,0.25589,0.48726,0.69395,0.20935,-0.00306,0.58351,0.37108,0.08142,0.14605,0.90002,0.59333}}
			\def\linearcoeffs{{1.07217,1.33547,16.39458,0.92996,1.62265,1.78554,1.21192,1.74783,1.80693,0.94224,1.32915,1.54949,0.83681,0.18166,1.4432,1.15855,0.4775,0.67013,1.69139,1.45363}}
			\def\constantcoeffs{{0.993714,0.837724,-15.998307,0.994525,0.227754,-0.888085,0.940974,-0.467311,-1.344585,0.996253,0.844586,0.458485,0.971545,0.500119,0.687641,0.968114,0.755463,0.891383,-0.08199,0.669272}}
			\foreach \i in {0,...,19}{
				\draw[domain=-2.0:2.0,smooth,variable=\x,gray] plot ({\x},{\quadraticcoeffs[\i]*\x*\x + \linearcoeffs[\i]*\x + \constantcoeffs[\i]});
			}
		\end{tikzpicture}
	\end{subfigure}
\end{figure}
\ \\
\begin{figure}[H]
	\caption{Approximation  with $\eps = 10^0, 10^{-1}, 10^{-2}$ of $\exp$ on $[-2, 2]$ from above by 2, 5, 32 paraboloids, respectively.}
	\label{fig:approx-exp-above}
	\centering
	\begin{subfigure}{0.32\textwidth}
		\centering
		\begin{tikzpicture}[scale=1.0]
			\clip (-2.1, -1.0) rectangle (2.1, 7.4);
			\ 
			\draw[->] (-2,0) -- (2,0) node[right] {$x$};
			\draw[->] (0,-0.6) -- (0,7.4) node[above] {$y$};
			\foreach \x in {-2.0,-1.5,...,2.0}
			\draw (\x,-0.05) -- (\x,0.05);
			\foreach \y in {0.5,1.0,...,7.0}
			\draw (-0.05,\y) -- (0.05,\y);
			\foreach \x/\label in {-2,-1,1,2}
			\draw (\x,0) node[below] {$\label$};
			\foreach \y in {1,...,7}
			\draw (0,\y) node[left] {$\y$};
			\draw[domain=-2.0:2.0,smooth,variable=\x,black, line width=2pt] plot ({\x},{exp(\x)});
			\def\quadraticcoeffs{{0.45183,0.76135}}
			\def\linearcoeffs{{1.97593,1.56324}}
			\def\constantcoeffs{{2.279876,1.217793}}
			\foreach \i in {0,...,1}{
				\draw[domain=-2.0:2.0,smooth,variable=\x,gray] plot ({\x},{\quadraticcoeffs[\i]*\x*\x + \linearcoeffs[\i]*\x + \constantcoeffs[\i]});
			}
		\end{tikzpicture}
	\end{subfigure}
	\hfill
	\begin{subfigure}{0.32\textwidth}
		\centering
		\begin{tikzpicture}[scale=1.0]
			\clip (-2.1, -1.0) rectangle (2.1, 7.4);
			\ 
			\draw[->] (-2,0) -- (2,0) node[right] {$x$};
			\draw[->] (0,-0.6) -- (0,7.4) node[above] {$y$};
			\foreach \x in {-2.0,-1.5,...,2.0}
			\draw (\x,-0.05) -- (\x,0.05);
			\foreach \y in {0.5,1.0,...,7.0}
			\draw (-0.05,\y) -- (0.05,\y);
			\foreach \x/\label in {-2,-1,1,2}
			\draw (\x,0) node[below] {$\label$};
			\foreach \y in {1,...,7}
			\draw (0,\y) node[left] {$\y$};
			\draw[domain=-2.0:2.0,smooth,variable=\x,black, line width=2pt] plot ({\x},{exp(\x)});
			\def\quadraticcoeffs{{0.49813,1.0284,0.7218,1.46874,2.16733}}
			\def\linearcoeffs{{1.79718,1.1337,1.62985,0.15076,-1.84929}}
			\def\constantcoeffs{{1.802177,1.008156,1.291587,1.212577,2.418805}}
			\foreach \i in {0,...,4}{
				\draw[domain=-2.0:2.0,smooth,variable=\x,gray] plot ({\x},{\quadraticcoeffs[\i]*\x*\x + \linearcoeffs[\i]*\x + \constantcoeffs[\i]});
			}
		\end{tikzpicture}
	\end{subfigure}
	\hfill
	\begin{subfigure}{0.32\textwidth}
		\centering
		\begin{tikzpicture}[scale=1.0]
			\clip (-2.1, -1.0) rectangle (2.1, 7.4);
			\ 
			\draw[->] (-2,0) -- (2,0) node[right] {$x$};
			\draw[->] (0,-0.6) -- (0,7.4) node[above] {$y$};
			\foreach \x in {-2.0,-1.5,...,2.0}
			\draw (\x,-0.05) -- (\x,0.05);
			\foreach \y in {0.5,1.0,...,7.0}
			\draw (-0.05,\y) -- (0.05,\y);
			\foreach \x/\label in {-2,-1,1,2}
			\draw (\x,0) node[below] {$\label$};
			\foreach \y in {1,...,7}
			\draw (0,\y) node[left] {$\y$};
			\draw[domain=-2.0:2.0,smooth,variable=\x,black, line width=2pt] plot ({\x},{exp(\x)});
			\def\quadraticcoeffs{{0.44374,1.22237,1.82983,0.59311,0.49338,0.95495,0.81247,2.93522,0.64349,0.72744,0.59735,1.07445,0.68534,0.44448,0.83712,1.33668,0.55346,0.64575,0.57854,2.09601,2.11377,0.45394,0.84863,0.82581,1.15187,0.73295,2.87955,2.80798,2.42549,2.43968,0.51898,1.51377}}
			\def\linearcoeffs{{1.81181,0.73728,-0.82521,1.74043,1.79794,1.2667,1.49324,-4.47859,1.69659,1.60554,1.73708,1.04518,1.65385,1.81171,1.4573,0.47697,1.76852,1.69442,1.75143,-1.62568,-1.68114,1.81018,1.44003,1.47397,0.88836,1.59885,-4.27671,-1.99836,-2.69211,-2.7397,1.78791,0.04019}}
			\def\constantcoeffs{{1.990497,1.025062,1.720163,1.535773,1.819657,1.035857,1.152703,4.605518,1.421922,1.26823,1.52551,1.000897,1.340007,1.987736,1.125989,1.088402,1.638194,1.417217,1.572037,2.256391,2.296273,1.952937,1.114494,1.137887,1.004894,1.259587,4.424277,2.037609,3.07132,3.109754,1.737329,1.253597}}
			\foreach \i in {0,...,31}{
				\draw[domain=-2.0:2.0,smooth,variable=\x,gray] plot ({\x},{\quadraticcoeffs[\i]*\x*\x + \linearcoeffs[\i]*\x + \constantcoeffs[\i]});
			}
		\end{tikzpicture}
	\end{subfigure}
\end{figure}

For the approximations of $x^3$ on $[-2, 2]$, we restrict the visualizations to $\eps = 10^0$ and $\eps = 10^{-1}$, since the large number of paraboloids required for a smaller accuracy does not reveal new insights. 
Further, the function $x^3$ is point-symmetric around the origin and thus limit the plots to the approximation direction from below, see~\Cref{fig:approx-x3-below}

\begin{figure}[H]
	\caption{Approximation  with $\eps = 10^0, 10^{-1}$ of $x^3$ on $[-2, 2]$ from below by 3 and 12 paraboloids, respectively.}
	\label{fig:approx-x3-below}
	\centering
	\begin{subfigure}{0.32\textwidth}
		\begin{tikzpicture}[scale=1.0]
			\clip (-2.1,-8.2) rectangle (2.1, 8.2);
			\ 
			\draw[->] (-2,0) -- (2,0) node[right] {$x$};
			\draw[->] (0,-8.2) -- (0,8.2) node[above] {$y$};
			\foreach \x in {-2,-1.5,...,2.0}
			\draw (\x,-0.05) -- (\x,0.05);
			\foreach \y in {-8.0,-7.5,...,8.0}
			\draw (-0.05,\y) -- (0.05,\y);
			\foreach \x/\label in {-2,-1,1,2}
			\draw (\x,0) node[below] {$\label$};
			\foreach \y in {-8,...,8}
			\draw (0,\y) node[left] {$\y$};
			\draw[domain=-2.0:2.0,smooth,variable=\x,black, line width=2pt] plot ({\x},{\x*\x*\x});
			\def\quadraticcoeffs{{-0.54782,-2.93587,1.29536}}
			\def\linearcoeffs{{2.36898,-1.97853,3.8375}}
			\def\constantcoeffs{{-1.070761,-0.387444,-5.506441}}
			\foreach \i in {0,...,2}{
				\draw[domain=-2.0:2.0,smooth,variable=\x,gray] plot ({\x},{\quadraticcoeffs[\i]*\x*\x + \linearcoeffs[\i]*\x + \constantcoeffs[\i]});
			}
		\end{tikzpicture}
	\end{subfigure}
	\begin{subfigure}{0.32\textwidth}
		\begin{tikzpicture}[scale=1.0]
			\clip (-2.1,-8.2) rectangle (2.1, 8.2);
			\ 
			\draw[->] (-2,0) -- (2,0) node[right] {$x$};
			\draw[->] (0,-8.2) -- (0,8.2) node[above] {$y$};
			\foreach \x in {-2,-1.5,...,2.0}
			\draw (\x,-0.05) -- (\x,0.05);
			\foreach \y in {-8.0,-7.5,...,8.0}
			\draw (-0.05,\y) -- (0.05,\y);
			\foreach \x/\label in {-2,-1,1,2}
			\draw (\x,0) node[below] {$\label$};
			\foreach \y in {-8,...,8}
			\draw (0,\y) node[left] {$\y$};
			\draw[domain=-2.0:2.0,smooth,variable=\x,black, line width=2pt] plot ({\x},{\x*\x*\x});
			\def\quadraticcoeffs{{0.6811,-3.49007,-4.48014,1.68347,1.23662,-2.24318,-1.57304,-0.97897,-0.33964,1.80802,0.1024,-2.90562}}
			\def\linearcoeffs{{3.56513,-3.53521,-6.49806,3.97495,3.85432,-0.50114,0.80835,1.78144,2.63153,3.99079,3.09978,-2.01628}}
			\def\constantcoeffs{{-3.594155,-1.11015,-3.075561,-6.783981,-5.237867,-0.029569,-0.091149,-0.521255,-1.378405,-7.250517,-2.210045,-0.410081}}
			\foreach \i in {0,...,11}{
				\draw[domain=-2.0:2.0,smooth,variable=\x,gray] plot ({\x},{\quadraticcoeffs[\i]*\x*\x + \linearcoeffs[\i]*\x + \constantcoeffs[\i]});
			}
		\end{tikzpicture}
	\end{subfigure}
\end{figure}

\newpage

\subsection{Overview of MINLP instances solved}
\label{subsec:overview-minlplib}

In~\Cref{tab:overview} we display an overview of the results on the MINLPLib.
For the evaluations, we excluded instance ex8\_4\_6 for \software{Gurobi} and eg\_disc\_s for \software{SCIP} since they caused numerical and/or memory errors only with the respective solver.
This is indicated by an asterisk.

 \begin{longtable}{lllccc}\toprule
 	instance & solver & type & run time in s & primal value & dual value\\ \hline\hline 
 	\multirow{6}{*}{batchdes} & \multirow{3}{*}{\software{Gurobi}} & both & 0.0 & 167427.7 & 167427.7\\
 	&  & orig & 0.0 & 167427.7 & 167427.7\\
 	&  & para & 0.0 & 140933.2 & 140937.2\\
 	& \multirow{3}{*}{\software{SCIP}} & both & 0.1 & 167426.7 & 167427.6\\
 	&  & orig & 0.1 & 167427.7 & 167427.7\\
 	&  & para & 0.1 & 140934.4 & 140937.1\\\hline
 	\multirow{6}{*}{contvar} & \multirow{3}{*}{\software{Gurobi}} & both & limit & 437878.5 & 1175110.6\\
 	&  & orig & limit & 498872.2 & inf\\
 	&  & para & limit & 433241.9 & 2091301.7\\
 	& \multirow{3}{*}{\software{SCIP}} & both & limit & 454624.4 & 809755.5\\
 	&  & orig & limit & 538588.7 & 809149.8\\
 	&  & para & limit & 459660.7 & 809567.8\\\hline
 	\multirow{6}{*}{eg\_all\_s} & \multirow{3}{*}{\software{Gurobi}} & both & 638.5 & 11.9 & 11.9\\
 	&  & orig & limit & 1.5 & 7.7\\
 	&  & para & limit & -2.6 & 13.9\\
 	& \multirow{3}{*}{\software{SCIP}} & both & limit & -3.3 & inf\\
 	&  & orig & 1775.0 & 7.7 & 7.7\\
 	&  & para & limit & -3.7 & 11.6\\\hline
 	\multirow{6}{*}{eg\_disc2\_s} & \multirow{3}{*}{\software{Gurobi}} & both & limit & -6.5 & inf\\
 	&  & orig & limit & -5.2 & 7.0\\
 	&  & para & limit & -8.4 & inf\\
 	& \multirow{3}{*}{\software{SCIP}} & both & limit & -7.8 & inf\\
 	&  & orig & limit & -1.1 & 6.3\\
 	&  & para & limit & -10.6 & inf\\\hline
 	\multirow{6}{*}{eg\_disc\_s} & \multirow{3}{*}{\software{Gurobi}} & both & limit & -6.4 & inf\\
 	&  & orig & limit & 2.0 & 7.5\\
 	&  & para & limit & -7.8 & inf\\
 	& \multirow{3}{*}{\software{SCIP}*} & both & limit & -7.7 & inf\\
 	&  & orig & limit & 3.6 & 5.8\\
 	&  & para & limit & -10.9 & 10.0\\\hline
 	\multirow{6}{*}{eg\_int\_s} & \multirow{3}{*}{\software{Gurobi}} & both & 989.0 & inf & inf\\
 	&  & orig & limit & -1.9 & inf\\
 	&  & para & limit & -3.1 & inf\\
 	& \multirow{3}{*}{\software{SCIP}} & both & limit & -3.2 & inf\\
 	&  & orig & 9085.1 & 6.5 & 6.5\\
 	&  & para & limit & -3.7 & inf\\\hline
 	\multirow{6}{*}{ex1222} & \multirow{3}{*}{\software{Gurobi}} & both & 0.0 & 1.1 & 1.1\\
 	&  & orig & 0.0 & 1.1 & 1.1\\
 	&  & para & 0.0 & 1.1 & 1.1\\
 	& \multirow{3}{*}{\software{SCIP}} & both & 0.1 & 1.1 & 1.1\\
 	&  & orig & 0.0 & 1.1 & 1.1\\
 	&  & para & 0.0 & 1.1 & 1.1\\\hline
 	\multirow{6}{*}{ex14\_1\_3} & \multirow{3}{*}{\software{Gurobi}} & both & 0.2 & -0.0 & -0.0\\
 	&  & orig & 0.0 & -0.0 & -0.0\\
 	&  & para & 0.7 & -0.0 & -0.0\\
 	& \multirow{3}{*}{\software{SCIP}} & both & 0.1 & -0.0 & -0.0\\
 	&  & orig & 0.0 & -0.0 & -0.0\\
 	&  & para & 0.0 & -0.0 & -0.0\\\hline
 	\multirow{6}{*}{ex14\_1\_4} & \multirow{3}{*}{\software{Gurobi}} & both & 0.2 & -0.0 & -0.0\\
 	&  & orig & 0.1 & -0.0 & -0.0\\
 	&  & para & 0.3 & -0.0 & -0.0\\
 	& \multirow{3}{*}{\software{SCIP}} & both & 1.0 & -0.0 & -0.0\\
 	&  & orig & 0.1 & -0.0 & -0.0\\
 	&  & para & 1.0 & -0.0 & -0.0\\\hline
 	\multirow{6}{*}{ex3pb} & \multirow{3}{*}{\software{Gurobi}} & both & 0.0 & 68.0 & 68.0\\
 	&  & orig & 0.0 & 68.0 & 68.0\\
 	&  & para & 0.0 & 68.0 & 68.0\\
 	& \multirow{3}{*}{\software{SCIP}} & both & 0.3 & 68.0 & 68.0\\
 	&  & orig & 0.2 & 68.0 & 68.0\\
 	&  & para & 0.2 & 68.0 & 68.0\\\hline
 	\multirow{6}{*}{ex8\_1\_1} & \multirow{3}{*}{\software{Gurobi}} & both & 0.1 & -2.0 & -2.0\\
 	&  & orig & 0.0 & -2.0 & -2.0\\
 	&  & para & 0.1 & -2.0 & -2.0\\
 	& \multirow{3}{*}{\software{SCIP}} & both & 0.1 & -2.0 & -2.0\\
 	&  & orig & 0.0 & -2.0 & -2.0\\
 	&  & para & 0.1 & -2.0 & -2.0\\\hline
 	\multirow{6}{*}{ex8\_1\_2} & \multirow{3}{*}{\software{Gurobi}} & both & 6.2 & -1.1 & -1.1\\
 	&  & orig & 2.6 & -1.4 & -1.4\\
 	&  & para & 6.0 & -1.1 & -1.1\\
 	& \multirow{3}{*}{\software{SCIP}} & both & 0.5 & -1.1 & -1.1\\
 	&  & orig & 1.0 & -1.1 & -1.1\\
 	&  & para & 0.3 & -1.1 & -1.1\\\hline
 	\multirow{6}{*}{ex8\_2\_1b} & \multirow{3}{*}{\software{Gurobi}} & both & 0.1 & -979.2 & -979.2\\
 	&  & orig & 0.1 & -979.2 & -979.2\\
 	&  & para & 0.2 & -979.4 & -979.4\\
 	& \multirow{3}{*}{\software{SCIP}} & both & 0.6 & -979.2 & -979.2\\
 	&  & orig & 0.4 & -979.3 & -979.2\\
 	&  & para & 0.5 & -979.4 & -979.4\\\hline
 	\multirow{6}{*}{ex8\_2\_4b} & \multirow{3}{*}{\software{Gurobi}} & both & 0.1 & -1197.2 & -1197.1\\
 	&  & orig & 0.1 & -1197.2 & -1197.1\\
 	&  & para & 0.1 & -1197.4 & -1197.4\\
 	& \multirow{3}{*}{\software{SCIP}} & both & 0.8 & -1197.1 & -1197.1\\
 	&  & orig & 0.4 & -1197.1 & -1197.1\\
 	&  & para & 0.7 & -1197.4 & -1197.4\\\hline
 	\multirow{6}{*}{ex8\_4\_6} & \multirow{3}{*}{\software{Gurobi}*} & both & 1.8 & 0.1 & 0.1\\
 	&  & orig & 92.5 & 0.0 & 0.0\\
 	&  & para & 68.0 & 0.6 & 0.6\\
 	& \multirow{3}{*}{\software{SCIP}} & both & 426.3 & 0.0 & 0.0\\
 	&  & orig & limit & -3e+09 & 0.0\\
 	&  & para & limit & 0.0 & 0.0\\\hline
 	\multirow{6}{*}{ex8\_4\_7} & \multirow{3}{*}{\software{Gurobi}} & both & 1507.3 & 28.9 & 28.9\\
 	&  & orig & 11.0 & 28.8 & 28.8\\
 	&  & para & 11099.4 & 26.9 & 26.9\\
 	& \multirow{3}{*}{\software{SCIP}} & both & limit & 28.4 & 29.0\\
 	&  & orig & limit & 25.5 & 29.0\\
 	&  & para & limit & 24.3 & 27.0\\\hline
 	\multirow{6}{*}{feedtray} & \multirow{3}{*}{\software{Gurobi}} & both & limit & -43.3 & -13.6\\
 	&  & orig & limit & -45.3 & inf\\
 	&  & para & limit & -64.1 & -14.9\\
 	& \multirow{3}{*}{\software{SCIP}} & both & limit & -68.6 & -13.4\\
 	&  & orig & limit & -68.7 & -13.4\\
 	&  & para & limit & -68.7 & -29.1\\\hline
 	\multirow{6}{*}{ghg\_1veh} & \multirow{3}{*}{\software{Gurobi}} & both & 8.6 & 7.8 & 7.8\\
 	&  & orig & 16.0 & 7.8 & 7.8\\
 	&  & para & 10.0 & 7.8 & 7.8\\
 	& \multirow{3}{*}{\software{SCIP}} & both & 14.6 & 7.8 & 7.8\\
 	&  & orig & 19.4 & 7.8 & 7.8\\
 	&  & para & 25.7 & 7.8 & 7.8\\\hline
 	\multirow{6}{*}{ghg\_2veh} & \multirow{3}{*}{\software{Gurobi}} & both & limit & 7.5 & 7.8\\
 	&  & orig & 5390.8 & 6.9 & 7.8\\
 	&  & para & limit & 7.4 & 7.8\\
 	& \multirow{3}{*}{\software{SCIP}} & both & limit & 7.6 & 7.8\\
 	&  & orig & limit & 7.4 & 7.8\\
 	&  & para & limit & 7.6 & 7.8\\\hline
 	\multirow{6}{*}{ghg\_3veh} & \multirow{3}{*}{\software{Gurobi}} & both & limit & 6.5 & 7.8\\
 	&  & orig & limit & 5.8 & 7.8\\
 	&  & para & limit & 6.5 & 7.8\\
 	& \multirow{3}{*}{\software{SCIP}} & both & limit & 6.5 & 7.8\\
 	&  & orig & limit & 5.1 & 7.8\\
 	&  & para & limit & 5.9 & 7.7\\\hline
 	\multirow{6}{*}{inscribedsquare01} & \multirow{3}{*}{\software{Gurobi}} & both & 0.9 & 1.0 & 1.0\\
 	&  & orig & 0.4 & 1.0 & 1.0\\
 	&  & para & 1.0 & 1.0 & 1.0\\
 	& \multirow{3}{*}{\software{SCIP}} & both & 4.0 & 1.0 & 1.0\\
 	&  & orig & 0.7 & 1.0 & 1.0\\
 	&  & para & 4.6 & 1.0 & 1.0\\\hline
 	\multirow{6}{*}{inscribedsquare02} & \multirow{3}{*}{\software{Gurobi}} & both & 1.5 & 1.0 & 1.0\\
 	&  & orig & 0.8 & 1.0 & 1.0\\
 	&  & para & 3.7 & 1.0 & 1.0\\
 	& \multirow{3}{*}{\software{SCIP}} & both & 6.2 & 1.0 & 1.0\\
 	&  & orig & 1.7 & 1.0 & 1.0\\
 	&  & para & 15.1 & 1.0 & 1.0\\\hline
 	\multirow{6}{*}{inscribedsquare03} & \multirow{3}{*}{\software{Gurobi}} & both & 2.2 & 23.8 & 23.8\\
 	&  & orig & 1.2 & 23.8 & 23.8\\
 	&  & para & 2.9 & 24.0 & 24.0\\
 	& \multirow{3}{*}{\software{SCIP}} & both & 11.6 & 23.8 & 23.8\\
 	&  & orig & 4.5 & 23.8 & 23.8\\
 	&  & para & 15.8 & 24.0 & 24.0\\\hline
 	\multirow{6}{*}{kriging\_peaks-red020} & \multirow{3}{*}{\software{Gurobi}} & both & 14.4 & 0.4 & 0.4\\
 	&  & orig & 4.0 & 0.4 & 0.4\\
 	&  & para & 14.2 & 0.4 & 0.4\\
 	& \multirow{3}{*}{\software{SCIP}} & both & 31.3 & 0.4 & 0.4\\
 	&  & orig & 773.7 & 0.4 & 0.4\\
 	&  & para & 19.2 & 0.4 & 0.4\\\hline
 	\multirow{6}{*}{lnts100} & \multirow{3}{*}{\software{Gurobi}} & both & limit & 0.6 & 0.6\\
 	&  & orig & limit & 0.5 & 0.6\\
 	&  & para & limit & 0.6 & 0.8\\
 	& \multirow{3}{*}{\software{SCIP}} & both & limit & 0.6 & 0.6\\
 	&  & orig & limit & 0.5 & 0.6\\
 	&  & para & 3889.5 & 0.6 & 0.6\\\hline
 	\multirow{6}{*}{lnts200} & \multirow{3}{*}{\software{Gurobi}} & both & limit & 0.5 & inf\\
 	&  & orig & limit & 0.5 & 0.6\\
 	&  & para & limit & 0.6 & 0.6\\
 	& \multirow{3}{*}{\software{SCIP}} & both & limit & 0.5 & 0.6\\
 	&  & orig & limit & 0.5 & 0.6\\
 	&  & para & limit & 0.5 & 0.6\\\hline
 	\multirow{6}{*}{lnts400} & \multirow{3}{*}{\software{Gurobi}} & both & limit & 0.5 & inf\\
 	&  & orig & limit & 0.5 & 0.6\\
 	&  & para & limit & 0.6 & 0.7\\
 	& \multirow{3}{*}{\software{SCIP}} & both & limit & 0.5 & inf\\
 	&  & orig & limit & 0.5 & 0.6\\
 	&  & para & limit & 0.5 & inf\\\hline
 	\multirow{6}{*}{lnts50} & \multirow{3}{*}{\software{Gurobi}} & both & 13371.4 & 0.6 & 0.6\\
 	&  & orig & 5811.5 & 0.6 & 0.6\\
 	&  & para & limit & 0.6 & 0.6\\
 	& \multirow{3}{*}{\software{SCIP}} & both & limit & 0.6 & 0.6\\
 	&  & orig & limit & 0.5 & 0.6\\
 	&  & para & 51.1 & 0.6 & 0.6\\\hline
 	\multirow{6}{*}{mathopt6} & \multirow{3}{*}{\software{Gurobi}} & both & 0.2 & -3.3 & -3.3\\
 	&  & orig & 0.1 & -3.3 & -3.3\\
 	&  & para & 0.1 & -3.3 & -3.3\\
 	& \multirow{3}{*}{\software{SCIP}} & both & 1.5 & -3.3 & -3.3\\
 	&  & orig & 0.5 & -3.3 & -3.3\\
 	&  & para & 0.5 & -3.3 & -3.3\\\hline
 	\multirow{5}{*}{polygon100} & \multirow{3}{*}{\software{Gurobi}} & both & limit & -1.6 & -0.7\\
 	&  & orig & limit & -1.6 & -0.8\\
 	&  & para & limit & -2.2 & -0.8\\
 	& \multirow{2}{*}{\software{SCIP}}  & both & limit & -1.6 & 0.0 \\
 	& & orig & limit & -26.6 & -0.8\\
 	&  & para & limit & -43.9 & -0.1\\\hline
 	\multirow{6}{*}{polygon25} & \multirow{3}{*}{\software{Gurobi}} & both & limit & -1.6 & -0.7\\
 	&  & orig & limit & -1.6 & -0.8\\
 	&  & para & limit & -1.9 & -0.7\\
 	& \multirow{3}{*}{\software{SCIP}} & both & limit & -1.4 & -0.7\\
 	&  & orig & limit & -4.4 & -0.8\\
 	&  & para & limit & -1.9 & -0.7\\\hline
 	\multirow{6}{*}{polygon50} & \multirow{3}{*}{\software{Gurobi}} & both & limit & -1.6 & -0.7\\
 	&  & orig & limit & -1.6 & -0.8\\
 	&  & para & limit & -2.1 & -0.8\\
 	& \multirow{3}{*}{\software{SCIP}} & both & limit & -1.6 & -0.7\\
 	&  & orig & limit & -10.4 & -0.8\\
 	&  & para & limit & -2.8 & -0.8\\\hline
 	\multirow{6}{*}{polygon75} & \multirow{3}{*}{\software{Gurobi}} & both & limit & -1.6 & -0.7\\
 	&  & orig & limit & -1.6 & -0.8\\
 	&  & para & limit & -2.2 & -0.8\\
 	& \multirow{3}{*}{\software{SCIP}} & both & limit & -1.6 & 0.0\\
 	&  & orig & limit & -17.4 & -0.8\\
 	&  & para & limit & -3.1 & -0.9\\\hline
 	\multirow{6}{*}{pooling\_epa1} & \multirow{3}{*}{\software{Gurobi}} & both & 8.6 & -280.8 & -280.8\\
 	&  & orig & 3.6 & -280.8 & -280.8\\
 	&  & para & 6.1 & -291.8 & -291.7\\
 	& \multirow{3}{*}{\software{SCIP}} & both & 142.1 & -280.8 & -280.8\\
 	&  & orig & 53.9 & -280.8 & -280.8\\
 	&  & para & 9.4 & -291.8 & -291.7\\\hline
 	\multirow{6}{*}{pooling\_epa2} & \multirow{3}{*}{\software{Gurobi}} & both & 126.6 & -4567.7 & -4567.4\\
 	&  & orig & 224.9 & -4567.7 & -4567.3\\
 	&  & para & 102.1 & -4567.6 & -4567.2\\
 	& \multirow{3}{*}{\software{SCIP}} & both & 2797.2 & -4567.4 & -4567.4\\
 	&  & orig & limit & -4637.3 & -4567.4\\
 	&  & para & 2047.9 & -4567.7 & -4567.4\\\hline
 	\multirow{6}{*}{pooling\_epa3} & \multirow{3}{*}{\software{Gurobi}} & both & limit & -14998.5 & -14961.6\\
 	&  & orig & limit & -14998.6 & -14963.5\\
 	&  & para & limit & -14998.6 & -14963.0\\
 	& \multirow{3}{*}{\software{SCIP}} & both & limit & -14998.6 & inf\\
 	&  & orig & limit & -14998.6 & -14936.8\\
 	&  & para & limit & -14998.6 & -14535.8\\\hline
 	\multirow{6}{*}{synthes2} & \multirow{3}{*}{\software{Gurobi}} & both & 0.1 & 73.0 & 73.0\\
 	&  & orig & 0.0 & 73.0 & 73.0\\
 	&  & para & 0.1 & 72.9 & 72.9\\
 	& \multirow{3}{*}{\software{SCIP}} & both & 0.2 & 73.0 & 73.0\\
 	&  & orig & 0.1 & 73.0 & 73.0\\
 	&  & para & 0.6 & 72.9 & 72.9\\\hline
 	\multirow{6}{*}{synthes3} & \multirow{3}{*}{\software{Gurobi}} & both & 0.1 & 68.0 & 68.0\\
 	&  & orig & 0.0 & 68.0 & 68.0\\
 	&  & para & 0.0 & 68.0 & 68.0\\
 	& \multirow{3}{*}{\software{SCIP}} & both & 0.3 & 68.0 & 68.0\\
 	&  & orig & 0.2 & 68.0 & 68.0\\
 	&  & para & 0.3 & 68.0 & 68.0\\\hline
 	\multirow{6}{*}{t1000} & \multirow{3}{*}{\software{Gurobi}} & both & 1.5 & 0.0 & 0.0\\
 	&  & orig & 1.3 & 0.0 & 0.0\\
 	&  & para & 0.5 & 0.0 & 0.0\\
 	& \multirow{3}{*}{\software{SCIP}} & both & 0.2 & 0.0 & 0.0\\
 	&  & orig & 915.5 & -0.0 & -0.0\\
 	&  & para & 0.6 & 0.0 & 0.0\\
 	\bottomrule
 	\caption{Overview of all MINLPLib instances solved.}
 	\label{tab:overview}
 \end{longtable}

\end{document}